\documentclass[10pt]{amsart}
\usepackage{amssymb,amsmath,amsthm}
\usepackage{mathrsfs,dsfont,a4wide}
\usepackage{amssymb}
\usepackage{graphicx}
\usepackage{amsmath}
\usepackage[colorlinks=true]{hyperref}     
\usepackage{manfnt}
\usepackage{xcolor}
\hypersetup{linkcolor=black, citecolor=black}
\newtheorem{theorem} {Theorem}[section]
\newtheorem{lemma}[theorem]{Lemma}

\newtheorem{proposition}[theorem]{Proposition}

\newtheorem{remark}[theorem]{Remark}

\newtheorem{definition}[theorem]{Definition}

\numberwithin{equation}{section}

\def\N{{\mathbb N}}

\def\R{{\mathbb R}}

\def\pmb#1{\setbox0=\hbox{#1}%
\kern-.025em\copy0\kern-\wd0
\kern.05em\copy0\kern-\wd0
\kern-.025em\raise.0433em\box0
}

    \def\e{\varepsilon}

\def\CC{\subset\kern-2pt\subset}
\def\qual{{\bf \Bigg[\kern-4.5pt\Bigg[}}
\def\quar{{\bf \Bigg]\kern-4.5pt\Bigg]}}
\def\quals{{\bf \big[\kern-5pt\big[}}
\def\quars{{\bf \big]\kern-5pt\big]}}
\def\qualm{{\bf \bigg[\kern-4pt\bigg[}}
\def\quarm{{\bf \bigg]\kern-4pt\bigg]}}

\def\sbh {S\kern-2pt B\kern-2pt H}  
\def\sbv {S\kern-2pt B\kern-1pt V}  
             
\def\gsbv {G\kern-1pt S\kern-1.5pt BV}

\newcommand{\as}{{\mathcal A}}
\newcommand{\hs}{{\mathcal H}}

\newcommand{\fs}{{\mathcal F}}
\newcommand{\leb}{{\mathcal L}}

\newcommand{\ms}{{\mathcal M}}

\newcommand{\Ss}{{\mathcal S}}

\newcommand{\hn}{{\mathcal{H}}^{N-1}}

\newcommand{\Om}{\Omega}
\newcommand{\Omb}{\overline{\Omega}}

\newcommand{\weak}{\rightharpoonup}

\newcommand{\Prob}[2]{\widetilde{Per}(#1, #2)}
\newcommand{\tsub}{\;\widetilde{\subseteq}\;}
\newcommand{\teq}{\;\widetilde{=}\;}
\newcommand{\lkb}{\tilde \lambda_{k,\beta}}

\definecolor{verde}{RGB}{20,150,100}

\begin{document}
\title[Minimization of the Robin eigenvalues with perimeter constraint]{Minimization of the $k$-th eigenvalue of the Robin-Laplacian with perimeter constraint}
\author[S. Cito]{Simone Cito}
\address[Simone Cito]{Dipartimento di Matematica e Fisica ``E. De Giorgi'', Universit\`a del Salento, Via per Arnesano, 73100 Lecce, Italy.}
\email[S. Cito]{simone.cito@unisalento.it}

\author[A. Giacomini]
{Alessandro Giacomini}
\address[Alessandro Giacomini]{DICATAM, Sezione di Matematica, Universit\`a degli Studi di Brescia, Via Branze 43, 25133 Brescia, Italy}
\email[A. Giacomini]{alessandro.giacomini@unibs.it}

\begin{abstract}
In this paper we address the problem of the minimization of the $k$-th Robin eigenvalue $\lambda_{k,\beta}$ with parameter $\beta>0$ among bounded open Lipschitz sets with prescribed perimeter.
The perimeter constraint allows us to naturally generalize the problem to a setting involving more general admissible geometries made up of sets of finite perimeter with inner cracks, along with a suitable generalization of the Robin-Laplacian operator with properties which look very similar to those of the classical setting. Within this extended framework we establish existence of minimizers, and prove that the associated eigenvalue coincides with the infimum of those achieved by regular domains.
\end{abstract}
\maketitle
\tableofcontents

\section{Introduction}
Given $\Om\subset\R^N$ open, bounded and with a Lipschitz boundary, $\lambda\in\R$ is said to be an eigenvalue of the Laplace operator under Robin (or Fourier) boundary conditions with constant $\beta>0$ if there exists a nontrivial $u\in W^{1,2}(\Om)$ such that
$$
\begin{cases}
-\Delta u=\lambda u&\text{in }\Om\\
\frac{\partial u}{\partial \nu}+\beta u=0&\text{on }\partial \Om,
\end{cases}
$$
which in the weak sense means
\begin{equation}
\label{eq:eig-intr}
\forall \varphi \in W^{1,2}(\Om)\,:\, \int_\Om \nabla u\cdot \nabla \varphi\,dx+\beta\int_{\partial \Om}u\varphi\,d\hn=\lambda\int_\Om u\varphi\,dx.
\end{equation}
Here $\nu$ denotes the outer normal to $\partial\Om$, while $\hn$ stands for the Hausdorff $(N-1)$-dimensional measure on $\R^N$, which coincides with the usual area measure on regular hypersurfaces.
\par
It is known that $\Om$ admits a positively diverging sequence of eigenvalues $0<\lambda_{1,\beta}(\Om)\le\lambda_{2,\beta}(\Om)\le \dots \le \lambda_{k,\beta}(\Om)\le \dots \to +\infty$,  which are given (counting multiplicity) by the min-max formula
\begin{equation}
\label{eq:minmaxintr}
\lambda_{k,\beta}(\Om)=\min_{V \in \Ss_k} \max_{u\in V, u\not=0} \frac{\int_\Om |\nabla u|^2\,dx+\beta\int_{\partial \Om}u^2\,d\hn}{\int_\Om u^2\,dx},
\end{equation}
where $\Ss_k$ denotes the family of vectorial subspaces of $W^{1,2}(\Om)$ with dimension $k$. The quantity appearing in \eqref{eq:minmaxintr} is the so called {\it Rayleigh quotient} $R_\beta$, and it involves a boundary term. 
\par
Shape optimization problems involving Robin eigenvalues have been widely studied in recent years (see for example \cite[Chapter 4]{Henrot} for an overview): the main difference with respect to the much more studied problems involving Dirichlet eigenvalues is due to the fact that $\lambda_{k,\beta}$ fails to be monotone under inclusion and does not enjoy simple rescaling properties, essentially because of the presence of the boundary integral in the Rayleigh quotient.
\par
The minimization of the first eigenvalue under a measure constraint leads to the so called Faber-Krahn inequality for the Robin-Laplacian: minimizers are balls, and this has been established only quite recently by Bossel \cite{Bossel1} in 1986 for two dimensional smooth domains, and by Daners \cite{Da} in 2006 for general $N$ dimensional Lipschitz domains. In the case $k=2$, Kennedy \cite{Kennedy} proved that optimal domains are the union of two congruent balls. For $k\ge 3$ the problem is open, and some advances have been achieved by Bucur and the second author in \cite{BucGiac-lambdak} employing techniques from {\it free discontinuity problems} which are described below. 
\par
In this paper we address the problem of minimizing $\lambda_{k,\beta}$ under a {\it perimeter constraint}, namely we study
\begin{equation}\label{eq:pbintr}
\min\left\{\lambda_{k,\beta}(\Omega):\Omega\subset\mathbb{R}^N\ \text{is a bounded Lipschitz domain with } \hn(\partial\Omega)=p \right\}
\end{equation}
where $p>0$. The existence and regularity of minimizers for problem \eqref{eq:pbintr} in the class of convex domains have been studied by the first author in \cite{cit18}. Allowing more general geometries, existence of minimizers is open as long as $k\ge 2$ (for $k=1$ minimizers are still balls). Our aim is to generalize problem \eqref{eq:pbintr} to a larger class of geometries in order to gain existence of optimal domains. The presence of the perimeter allows us to deal with the problem in a more geometrical way with respect to \cite{BucGiac-lambdak}, still employing free discontinuity arguments but in a much more simplified way (also in the Dirichlet case, the perimeter has a ``regularizing'' effect on the problem: under a measure constraint, existence is available within the class of {\it quasi-open} sets, and optimal domains are known to be bounded and of finite perimeter (see \cite{Buc-k}), while under a perimeter constraint De Philippis and Velichkov \cite{de2014existence} proved that optimal domains are open and with a fairly smooth boundary).

\vskip10pt
The issue of minimizing $\lambda_{k,\beta}$ for $k\ge 3$ under a measure constraint has been addressed by Bucur and the second author in \cite{BucGiac-lambdak}, by generalizing the {\it free discontinuity approach} developed in \cite{BucGiac} and \cite{bugi2015} (see also \cite{BuGiVenant} and \cite{BGT2019}) to deal with the optimization of the first eigenvalue
and associated variants (for example the {\it torsion} of the domain, leading to the so called de Saint-Venant inequality).
\par
Roughly speaking, in the case of the first eigenvalue, the free discontinuity approach can be summarized as follows. One replaces the dependence on a domain $\Omega$ with the dependence on a function $u$ belonging to a suitable class of {\it functions of bounded variation}, whose support will be identified with $\Omega$ and whose jump set will play the role of $\partial\Omega$. The {\it Rayleigh quotient} gives rise to a free discontinuity functional for $u$ of the form
\begin{equation}
\label{eq:ray-u}
R_\beta(u):=\frac{\int_{\R^N} |\nabla u|^2\,dx+\beta\int_{J_u}[\gamma_l^2(u)+\gamma_r^2(u)]\,d\hn}{\int_{\R^N} u^2\,dx},
\end{equation}
where $\gamma_l(u)$ and $\gamma_r(u)$ stands for the traces of $u$ on both sides of its jump set $J_u$, with respect to a given orientation. The minimization of $\lambda_{1,\beta}(\Om)$ is replaced by the minimization of $R_\beta(u)$ under a measure constraint for the support $supp(u)$. Existence of minimizers follows by using compactness and lower semicontinuity properties of free discontinuity functionals: the Faber-Krahn inequality is established (and its validity extends to more general geometries) by showing that minimizers $u$ have support equal to a ball.
\par 
For higher order eigenvalues, the main idea of \cite{BucGiac-lambdak} is to reinterpret the min-max characterization \eqref{eq:minmaxintr} in the free discontinuity setting, by replacing the $k$-dimensional subspaces $V_k$ with vector valued functions $u$ whose components are linearly independent and belong again to a suitable class of vector valued functions of bounded variation $\fs_k$. The problem is generalized to the minimization of $\max R_\beta$ in $\fs_k$ under a measure constraint, with $R_\beta$ given in \eqref{eq:ray-u}, and the max being computed on the space generated by the components of the functions. The intuitive idea behind the approach is that given a minimizer $u$, its support should provide the optimal domain $\Om$, the space generated by the components being the optimal subspace in the min-max characterization \eqref{eq:minmaxintr}. The generalized problem on $\fs_k$ can be seen as a {\it relaxation} of the original one in the following sense: the minimum value turns out to be the infimum of $\lambda_{k,\beta}$ in the class of Lipschitz domains. The same approach has been used by Nahon in \cite{Na} to deal with the minimization of more general functionals whose prototype is $\Om \mapsto \lambda_{1,\beta}(\Om)+\dots+\lambda_{k,\beta}(\Om)$ under a measure constraint: in this case, the supports of minimizers turn out to be open sets with boundary of finite $\hn$-measure.
\vskip10pt
In the present paper, as mentioned above, we deal with the optimization problem \eqref{eq:pbintr} still employing  free discontinuity arguments as in \cite{BucGiac-lambdak}, but the perimeter constraint permits us to  reformulate the problem in a setting which retains a more geometrical flavor. This is because the constraint suggests  that the class of sets of {\it finite perimeter}, which enjoy strong compactness and structural properties, should be naturally involved.
\par
In Section \ref{sec:adm} we generalize the Robin-Laplacian eigenvalue problem to geometries $(\Om,\Gamma)$, where $\Om\subset\R^N$ is a set of finite perimeter and finite volume, while $\Gamma\subset \Om^1$ is a ($\hn$-countably) rectifiable set  with finite $\hn$-measure: here $\Om^1$ stands for the family of points of density $1$ for $\Om$. The idea behind the choice of these geometries is that open bounded domains are naturally replaced by sets of finite perimeter. 
Moreover being interested in variational problems, letting thus the domains vary, it is natural to keep into account possible ``inner'' boundaries, which may arise as a degeneration of inner holes, or by a folding of outer boundaries: the rectifiable set $\Gamma$, which can be seen as a {\it crack} inside $\Om$, is introduced precisely for this purpose. 
\par
In order to extend the Robin-Laplacian boundary value problem to these irregular geometries, we generalize some ideas introduced in \cite{bgt2} in the context of two dimensional open sets with rectifiable topological boundary. In particular we replace the usual Sobolev space $W^{1,2}$ with
\begin{multline*}
\Theta(\Omega,\Gamma):=\Bigg\{u\in SBV(\R^N): u=0\ \text{a.e. in $\Omega^c$},  \nabla u\in L^2(\R^N), J_u\subseteq\partial^*\Omega\cup\Gamma\ \text{up to }\\
\text{ $\hn$-negligible sets, and }\int_{\partial^*\Omega\cup\Gamma}\left[\gamma_l(u)^2+\gamma_r(u)^2\right]\:d\hn<+\infty \Bigg\},
\end{multline*}
where $\partial^*\Om$ denotes the reduced boundary of $\Om$, $SBV(\R^N)$ is the space of {\it special functions of bounded variation} in $\R^N$, and $\gamma_l(u),\gamma_r(u)$ stand for the traces of $u$ from both sides of the rectfiable set $\partial^*\Omega\cup\Gamma$ with respect to (any) given orientation (see Sections \ref{sec:bv} and \ref{sec:per}). 
\par
The weak formulation of \eqref{eq:eig-intr} for the eigenvalue problem is generalized to (see Sections \ref{sec:gen-rob} and \ref{sec:gen-lbk})
\begin{equation}
\label{eq:eig-intr2}
\forall \varphi\in \Theta(\Om,\Gamma)\,:\,  \int_\Om \nabla u\cdot \nabla \varphi\,dx+\beta\int_{\partial^* \Om\cup \Gamma}[\gamma_l(u)\gamma_l(\varphi)+\gamma_r(u)\gamma_r(\varphi)]\,d\hn=\tilde \lambda\int_\Om u\varphi\,dx.
\end{equation}
We prove that $\Theta(\Om,\Gamma)$ can be endowed with a complete scalar product, which permits to deal with \eqref{eq:eig-intr2} using classical arguments coming from the Hilbert space approach to boundary value problems. In particular it is guaranteed the existence of a diverging sequence of eigenvalues $\lkb(\Om,\Gamma)$ which admit the min-max characterization (counting multiplicity)
$$
\lkb(\Omega,\Gamma):=\min_{\overset{V\subset\Theta(\Omega,\Gamma)}{\dim V=k}}\max_{u\in V\setminus\left\{0\right\}} \frac{\displaystyle\int_{\Om}|\nabla u|^2\:dx+\beta\int_{\partial^*\Omega\cup\Gamma}[\gamma_l(u)^2+\gamma_r(u)^2]\:d\hn}{\displaystyle\int_{\Om}u^2\:dx}.
$$
Concerning the the perimeter constraint, we make use of the {\it generalized perimeter}
$$
\Prob{\Om}{\Gamma}:=Per(\Om)+2\hn(\Gamma),
$$
where $Per(\Om)$ stands for the usual perimeter of $\Om$. The definition of $\Prob{\Om}{\Gamma}$ is again suggested by the interpretation of the inner crack $\Gamma$ as degenerated holes or inwards folds of the outer boundary, so that its contribution to the generalized perimeter is given naturally by twice its area $\hn$. Related notions of perimeter have been considered by Cerf in \cite{Ce02} in the study of the lower semicontinuous envelope of the Hausdorff measure for the approximation by smooth sets, and by Henrot and Zucco in  \cite{HZ19} in relationship with the Minkowski content. 
\par
 We reformulate problem \eqref{eq:pbintr} in the form
\begin{equation}\label{eq:mainpb-intr}
\min\left\{\lkb(\Omega,\Gamma): (\Omega,\Gamma)\in\mathcal{A}(\R^N),\ \Prob{\Omega}{\Gamma}=p\right\}.
\end{equation}
Notice that for a regular bounded domain $\Om\subset\R^N$, we have that $(\Om,\emptyset)\in \as(\R^N)$ with 
$$
\lkb(\Om,\emptyset)=\lambda_{k,\beta}(\Om)\qquad\text{and}\qquad\Prob{\Om}{\emptyset}=Per(\Om)=\hn(\partial\Om).
$$
We thus see that \eqref{eq:mainpb-intr} is a natural generalization of the original problem \eqref{eq:pbintr}.
 \vskip10pt
 The main result of the paper (Theorem \ref{th:main1}) is that problem \eqref{eq:mainpb-intr} is well posed: in addition, the minimal eigenvalue $\lkb$ is equal to the infimum of $\lambda_{k,\beta}$
 on regular domains, so that the generalized problem  \eqref{eq:mainpb-intr} can be seen as a good {\it relaxation} of the original one \eqref{eq:pbintr}.
\par
Existence of optimal configurations is established by applying the direct method of the Calculus of Variations. A delicate point is given by the compactness of a minimizing sequence $(\Om_n,\Gamma_n)$, in particular on the side of the inner cracks $\Gamma_n$ (compactness for $\Om_n$, at least locally, is guaranteed by the properties of the perimeter). In this direction we employ a notion of variational convergence for rectifiable sets, called  $\sigma^2$-convergence (see Section \ref{sec:sigma2}), introduced by Dal Maso, Francfort and Toader in \cite{DMFT} to study existence of crack evolutions in finite elasticity. It is a notion of convergence which enjoys good compactness and lower semicontinuity properties, and turns out to be very natural in our context: it plays essentially the same role of the Hausdorff convergence for compact connected cracks with finite $\hs^1$ length in dimension two, providing a generalization of G\"ol\c ab semicontinuity theorem of the length to higher dimensions.

 \par
 On the basis of Theorem, \ref{th:main1} we can establish existence of minimizers for functionals involving only some eigenvalues (see Theorem \ref{thm:main3}), whose prototype is $ (\Om,\Gamma)\mapsto \tilde \lambda_{1,\beta}(\Om,\Gamma)+\dots+\tilde\lambda_{k,\beta}(\Om,\Gamma)$ or more generally $(\Om,\Gamma)\mapsto [\tilde \lambda^p_{1,\beta}(\Om,\Gamma)+\dots+\tilde\lambda^p_{k,\beta}(\Om,\Gamma)]^{1/p}$ with $p>1$.
 
 \vskip10pt
 The paper is organized as follows. In Section \ref{sec:prel} we fix the notation and recall the main definitions and properties concerning sets of finite perimeter, special functions of bounded variation, and the variational $\sigma^2$-convergence for rectifiable sets. In Section \ref{sec:adm} we define the family of admissible configurations $(\Om,\Gamma)\in \as(\R^N)$, and generalize the Robin-Laplacian boundary value problem to those geometries, defining in  particular the generalized eigenvalues $\lkb(\Om,\Gamma)$. The main results of the paper are stated in Section \ref{sec:main}. Section \ref{sec:tech} collects some technical compactness, lower semicontinuity and approximation properties concerning the admissible configurations and their associated functional spaces. The proof of the main results is given in Section \ref{sec:proofs}.
 \par

\section{Notation and preliminaries}
\label{sec:prel}

\subsection{Basic notation}
If $E \subseteq \R^N$, we will denote with $|E|$ its $N$-dimensional Lebesgue measure, and by $\hn(E)$ its $(N-1)$-dimensional Hausdorff measure: we refer to \cite[Chapter 2]{EvansGariepy} for a precise definition, recalling that for sufficiently regular sets $\hn$ coincides with the usual area measure. Moreover, we denote by $E^c$ the complementary set of $E$, and by $1_E$ its characteristic function, i.e., $1_E(x)=1$ if $x \in E$, $1_E(x)=0$ otherwise. Finally, for $t\in [0,1]$ we will write $E^t$ for the points of density $t$ for $E$ (see \cite[Definition 3.60]{AFP}).
\par
If $A \subseteq \R^N$ is open and $1 \le p \le +\infty$, we denote by $L^p(A)$ the usual space of $p$-summable functions on $A$ with norm indicated by $\|\cdot\|_p$. $W^{k,2}(A)$ will stand for the Sobolev  space of functions in $L^2(A)$ whose derivatives up to order $k$ in the sense of distributions belongs to $L^2(A)$. Finally $\ms_b(A;\R^N)$ will denote the space of $\R^N$-valued Radon measures on $A$, which can be identified with the dual of $\R^N$-valued continuous functions on $A$ vanishing at the boundary. We will denote by $|\cdot|$ its total variation.
\par
We say that $\Gamma\subseteq \R^N$ is $\hn$-countably rectifiable if
$$
\Gamma=N\cup \bigcup_{i\in \N}\Gamma_i
$$
where $\hn(N)=0$ and $\Gamma_i\subseteq \ms_i$ are Borel sets, where $\ms_i$ is a $C^1$-hypersurface of $\R^N$. It is not restrictive to assume that the sets $\Gamma_i$ are mutually disjoint. In the rest of the paper, we will write simply rectifiable in place of $\hn$-countably rectifiable. If $\Gamma_1,\Gamma_2$ are rectifiable, we will write $\Gamma_1\tsub \Gamma_2$ if $\hn(\Gamma_1\setminus \Gamma_2)=0$, and $\Gamma_1\teq \Gamma_2$ if $\hn((\Gamma_1\setminus \Gamma_2)\cup (\Gamma_2\setminus \Gamma_1))=0$.

\subsection{Functions of bounded variation}
\label{sec:bv}
Let $A\subseteq \R^N$ be an open set. We say that $u \in BV(A)$ if $u \in L^1(A)$ and its derivative in the sense of distributions is a finite Radon measure on $A$, i.e., $Du \in \ms_b(A;\R^N)$. $BV(A)$ is called the space of {\it functions of bounded variation} on $A$. $BV(A)$ is a Banach space under the norm $\|u\|_{BV(A)}:=\|u\|_{L^1(A)}+\|Du\|_{\ms_b(A;\R^N)}$.  We call $|Du|(A):=\|Du\|_{\ms_b(A;\R^N)}$ the {\it total variation} of $u$. We refer the reader to \cite{AFP} for
an exhaustive treatment of the space $BV$.
\par
If $u\in BV(A)$, then the measure $Du$ can be decomposed canonically (and uniquely) as
$$
Du=D^au+D^ju+D^cu.
$$
The measure $D^au$ is the absolutely continuous part (with respect to the Lebesgue measure) of the derivative: the associated density is denoted by $\nabla u\in L^1(A;\R^N)$. The measure $D^ju$ is the {\it jump part} of the derivative and it turns out that
$$
D^ju=(u^+-u^-)\otimes \nu_u \,\hn\lfloor J_u.
$$
Here $J_u$ is the {\it jump set} of $u$, $\nu_u$ is the normal to $J_u$, while $u^\pm$ are the upper and lower approximate limits at $x$. It turns out that $J_u$ is a  rectifiable set: if we choose the orientation given by a normal vector field $\nu_u$ we have $\hn$-a.e. 
$$
u^+=\gamma_r(u)\qquad\text{and}\qquad u^-=\gamma_l(u)
$$
where $\gamma_r(u)$ and $\gamma_l(u)$ are the right and left traces of $u$ on the rectifiable set $J_u$, associated to the orientation given by $\nu_u$.
Finally $D^cu$ is called the {\it Cantor part} of the derivative, and it vanishes on sets which are $\sigma$-finite with respect to $\hn$. Clearly $D^ju+D^cu$ is the singular part $D^su$ of $Du$ with respect to $\leb^N$.
\par
The space $SBV(A)$ of Special Functions of Bounded Variation on $A$ is defined as
$$
SBV(A):=\{u\in BV(A)\,:\, D^cu=0\},
$$
i.e., it is composed of those functions of bounded variation with vanishing Cantor part.

The following result is contained in \cite[Theorem 3.1 and Remark 3.2]{CT} and provides a very important tool to approximate our relaxed configurations via smooth sets.

\begin{theorem}\label{teo:cor_toa}
Let $\Omega\subset\R^N$ be open and bounded with Lipschitz boundary, and let $q>1$. Let $u\in SBV(\Omega)\cap L^\infty(\Omega)$ be such that $\nabla u\in L^q(\Omega;\R^N)$ and $\hn(J_u)<+\infty$.
There exists $(u_n)_n$ such that the following items hold true for every $n\in\N$.
\begin{itemize}
    \item[(i)] $J_{u_n}$ is polyhedral in $\Omega$, i.e., $\hn((\overline{J_{u_n}}\setminus J_{u_n})\cap \Om)=0$, and $\overline{J_{u_n}}\cap \Om$ is given by intersection with $\Om$ of the union of a finite number of $(N-1)$-dimensional simplexes.
    \item[(ii)] $u_n\in W^{k,\infty}(\Omega\setminus \overline{J_{u_n}})$ for every $k\ge 1$.
    \item[(iii)] It holds
    $$u_n\to u\quad\text{strongly in $L^1(\Omega)$},$$
$$\nabla u_n\to \nabla u\quad\text{strongly in $L^q(\Omega;\R^{N})$},$$
and
$$
\limsup_{n\to+\infty}\int_{J_{v^j_n}\cap A}\varphi(x,v_n^+,v_n^-,\nu_{J_{v_n}})\:d\hn
\le\int_{J_{v}\cap A}\varphi(x,v^-,v^-,\nu_{J_{v}})\:d\hn
$$
for any open set $A\subset\subset \Omega$ and every upper continuous function $\varphi:\Omega\times\R\times\R\times\mathbb{S}^{N-1}\to[0,+\infty[$ such that $\varphi(x,a,b,\nu)=\varphi(x,b,a,-\nu)$. If $\varphi$ is locally bounded near the boundary, i.e., $\limsup_{(x_n,a_n,b_n,\nu_n)\to (x,a,b,\nu)}\varphi(x_n,a_n,b_n,\nu_n)\,d\hn<+\infty$ where $x_n\in \Om$ and $x\in \partial\Om$, then we can choose $A\subseteq \Om$.
\end{itemize}
\end{theorem}

\subsection{Sets of finite perimeter}
\label{sec:per}
Given $E\subseteq \R^N$ measurable and $A\subseteq \R^N$ open, we say that $E$ has finite perimeter in $A$ (or simply has finite perimeter if $A=\R^N$) if
$$
Per(E;A):=\sup\left\{\int_E div(\varphi)\,dx\,:\, \varphi\in C^\infty_c(A;\R^N), \|\varphi\|_\infty\le 1\right\}<+\infty.
$$
If $|E|<+\infty$, then $E$ has finite perimeter if and only if $1_E\in BV(\R^N)$.
It turns out that
$$
D1_E=\nu_E \hn\lfloor \partial^*E,\qquad Per(E;\R^N)=\hn(\partial^*E),
$$
where the rectifiable set $\partial^*E$ is called the {\it reduced boundary} of $E$, and $\nu_E$ is the associated inner approximate normal (see \cite[Section 3.5]{AFP}). It turns out that $\partial^*E\subseteq \partial E$, but the topological boundary can in general be much larger than the reduced one.

\par
The next proposition is the collection of two approximation results proved in \cite{CoTo}, and will prove particularly useful for our problem.

\begin{proposition}[\bf Interior approximation via smooth sets]\label{pro:comitorres}
Let $\mu$ be a Radon measure on $\R^N$ such that $\mu<<\hn$ and let $E\subset\R^N$ be a bounded set of finite perimeter. Let $u_k:=1_E*\rho_{\varepsilon_k}$, where $\rho_k$ is a regularizing kernel, and let $A_{k,t}:=\{u_k>t\}$. Then for a.e. $t\in (0,1)$, $A_{k,t}$ is a smooth set and, for a.e. $t\in (1/2,1)$, the sequence $(A_{k,t})_k$ provides an interior approximation of $E$, i.e., 
$$
\lim_{k\to+\infty}|\mu|(A_{k,t}\Delta E^1)=0
$$
and
$$
\lim_{k\to+\infty}\hn(\partial A_{k,t}\setminus E^1)=0.
$$
\end{proposition}

\subsection{A variational convergence for rectifiable sets}
\label{sec:sigma2} 
 We recall the notion of $\sigma^2$-convergence for rectifiable sets introduced in \cite{DMFT} to deal with problems in fracture mechanics. It is a variational notion of convergence for rectifiable sets which enjoy compactness and lower semicontinuity properties under uniform bound for the associated $\hn$ measure, which are very similar to that enjoyed by connected closed sets in $\R^2$ with respect to Hausdorff convergence in view of Go\c l\" ab semicontinuity theorem. The definition is based on the use of the space $SBV$. Recall the notation $\Gamma_1\tsub \Gamma_2$ and $\Gamma_1\teq \Gamma_2$ which denote inclusion and equality up to $\hn$-negligible sets.

\begin{definition}[\bf $\sigma^2$-convergence]
\label{def:sigma2}
Let $D\subset\R^N$ be open and bounded, and let $\Sigma_n,\Sigma\subset D$ be rectifiable sets such that $\hn(\Sigma_n),\hn(\Sigma)\le C$ for some $C>0$. We say that 
$$
\Sigma_n \to \Sigma \qquad\text{in the sense of $\sigma^2$-convergence}
$$ 
if the following two conditions are satisfied.
\begin{itemize}
\item[(a)] If $u_j\in SBV(D)$ with $J_{u_j}\tsub \Sigma_{n_j}$ for some sequence $n_j\to+\infty$, and $u\in SBV(D)$ are such that $\|u_j\|_\infty, \|u\|_\infty \le C$,
$$
u_j \to u\qquad\text{strongly in }L^1(D)
$$
and
$$
\nabla u_j \to \nabla u\qquad\text{weakly in }L^2(D;\R^N),
$$
then $J_{u}\tsub \Sigma$.

\item[(b)] There exist a function $u\in SBV(D)$ with $\nabla u\in L^2(D;\R^N)$ and a sequence $u_n \in SBV(D)$ with $\|u\|_\infty, \|u_n\|_\infty \le C$,
$$
u_n \to u\qquad\text{strongly in }L^1(D)
$$
and
$$
\nabla u_n \to \nabla u\qquad\text{weakly in }L^2(D;\R^N),
$$
such that  $J_u \teq\Sigma$ and $J_{u_n}\tsub \Sigma_{n}$ for every $n\in\N$.
\end{itemize}
\end{definition}

\begin{remark}
{\rm
Condition (a) guarantees that $\Sigma$ contains the jump sets of functions which are suitable limits of functions jumping on $\Sigma_n$. Condition (b) ensures that $\Sigma$ is the smallest set which enjoys this property. The notion of convergence introduced in \cite{DMFT} can indeed be generalized to an exponent $p\in ]1,+\infty[$: we will use only the case $p=2$.
}
\end{remark}

The following result compactness and lower semicontinuity result holds true, and will be fundamental for our analysis.

\begin{theorem}[\bf Compactness and lower semicontinuity for $\sigma^2$-convergence]
\label{thm:sigma2}
Let $D\subset\R^N$ be open and bounded. For every sequence $\Sigma_n\subset D$ of rectifiable sets such that $\hn(\Sigma_n)\le C$, there exist a rectifiable set $\Sigma\subset D$ and a subsequence $\Sigma_{n_k}$ such that
$$
\Sigma_{n_k} \to \Sigma \qquad\text{in the sense of $\sigma^2$-convergence.}
$$
Moreover we have
$$
\hn(\Sigma) \le \liminf_{n\to +\infty} \hn(\Sigma_n).
$$
\end{theorem}

\section{Admissible geometries and the associated generalization of the Robin-Laplacian}
\label{sec:adm}
 In this section we introduce the class $\as(\R^N)$ of admissible geometries, and extend to this setting the Robin-Laplacian boundary value problem.

\subsection{Admissible geometries}
\label{sec:geom}
 The precise definition of the class of admissible geometries we consider is the following. Recall that $\Om^1$ stands for the set of points of density one for $\Om$.
\begin{definition}[{\bf Admissible geometries}]
We say that the couple $(\Omega,\Gamma)$ is an admissible geometry, and we write $(\Omega,\Gamma)\in\mathcal{A}(\R^N)$, if $\Omega \subset \R^N$ is a set of finite perimeter with $|\Om|<+\infty$, and $\Gamma\subset \Omega^1$ is a rectifiable set with $\hn(\Gamma)<+\infty$. 
\end{definition}

\begin{remark}
{\rm
From a geometrical point of view, we think $(\Omega,\Gamma)\in\mathcal{A}(\R^N)$ as the set $\Om$ with an ``inner'' crack $\Gamma$: the domain is thus given essentially by the possibly irregular set $\Om\setminus \Gamma$. If $\Om\subset \R^N$ is Lipschitz, then $(\Om,\emptyset)\in \as(\R^N)$.
}
\end{remark}
In order to give a meaning to the Robin boundary value problem for the Laplacian on an admissible geometry $(\Om,\Gamma)$, and to define the associated eigenvalues, we need to define a functional space which can replace the Sobolev space $W^{1,2}$ on the (possibly irregular) set $\Om\setminus \Gamma$. 
\par

\begin{remark}
\label{rem:traces}
{\rm
Given a rectifiable set $K\subset \R^N$ with $\hn(K)<+\infty$, by definition, we may write
\begin{equation*}
K=N\cup  \bigcup_{i=0}^\infty K_i,
\end{equation*}
where $\hn(N)=0$, while for every $i\in \N$ the Borel sets $K_i$ are subsets of a $\mathcal{C}^1$ -manifold $\ms_i$ of dimension $N-1$, and $K_i\cap K_j=\emptyset$ for $i\neq j$. 
\par
It is not restrictive, up to reducing $\ms_i$, to assume that $\ms_i$ is orientable with associated normal vector filed $\nu_i$, and that two continuous trace operators from $BV(\R^N)$ to $L^1(\ms_i)$, the ``left'' and ``right'' traces, are defined. For every $u\in BV(\R^N)$, let us denote by $\gamma_l^i(u),\gamma_r^i(u)$  the ``left'' and ``right'' traces of $u$ on $K_i$, using the orientation associated to $\nu_i$. 
By general theory of BV functions it is known that
\begin{equation*}
Du\lfloor K=\sum_i [\gamma^i_r(u)-\gamma^i_l(u)]\nu_i \hn\lfloor K_i.
\end{equation*}
We define global ``left'' and ``right'' traces on the full $K$ by setting
\begin{equation*}
\gamma_l(v):=\sum_i \gamma_l^i(v)
\qquad\text{and}\qquad
\gamma_r(v):=\sum_i \gamma_r^i(v).
\end{equation*}
}
\end{remark}

The functional space we are looking for is the following.

\begin{definition}[\bf Admissible functions]
Let $(\Omega,\Gamma)\in\mathcal{A}(\R^N)$. We set
\begin{multline*}
\Theta(\Omega,\Gamma):=\Bigg\{u\in SBV(\R^N): u=0\ \text{a.e. in $\Omega^c$}, \nabla u\in L^2(\R^N), J_u\tsub\partial^*\Omega\cup\Gamma\\
\text{ and }\int_{\partial^*\Omega\cup\Gamma}\left[\gamma^2_l(u)+\gamma^2_r(u)\right]\:d\hn<+\infty \Bigg\},
\end{multline*}
 where $\gamma_l(u),\gamma_r(u)$ are the left and right traces of $u$ on $\partial^*\Om\cup \Gamma$ defined according to Remark \ref{rem:traces}.
\end{definition}

\begin{remark}
{\rm
Notice that if $\Om\subset\R^N$ is open, bounded and with a Lipschitz boundary, $(\Om,\emptyset)\in \as(\R^N)$ and $\Theta(\Om,\emptyset)$ is simply given by the extension to zero outside $\Om$ of the functions in the Sobolev space $W^{1,2}(\Om)$.
}
\end{remark}

The following lemma holds true.

\begin{lemma}
\label{lem:L2}
Let $(\Omega,\Gamma)\in\mathcal{A}(\R^N)$ and $u\in \Theta(\Om,\Gamma)$. Then $u\in L^2(\Om)$ with 
$$
\|u\|_{L^2(\Om)}\le C\left( \|\nabla u\|_{L^2(\Om;\R^N)}+\|\gamma_l(u)\|_{L^2(\partial^*\Om\cup \Gamma)}+\|\gamma_r(u)\|_{L^2(\partial^*\Om\cup \Gamma)} \right),
$$
where $C=C(N,|\Om|)$.
\end{lemma}

\begin{proof}
Clearly the space $\Theta(\Om,\Gamma)$ is closed under truncation. For every $n\in\N$ let us consider
$$
u_n:=\max\{\min\{u,n\},-n\} \in \Theta(\Om,\Gamma).
$$
By the chain rule in BV (see \cite[Theorem 3.96]{AFP}) we get $u^2_n\in BV(\R^N)$ with 
$$
\nabla(u_n)^2=2u_n\nabla u_n\qquad\text{and}\qquad \gamma_{l/r}(u_n^2)=\gamma_{l/r}(u_n)^2.
$$
Using the embedding of $BV(\R^N)$ into $L^{\frac{N}{N-1}}(\R^N)$ and since $u$ is supported in $\Om$ we can write employing also the Cauchy inequality
\begin{multline*}
\|u_n\|^2_{L^2(\Om)}\le |\Om|^{1/N}\|u^2_n\|_{L^{\frac{N}{N-1}}(\R^N)}\le C_N |\Om|^{\frac{1}{N}}|D(u_n^2)|(\R^N)\\
\le
C_N |\Om|^{\frac{1}{N}}\left( 2\int_\Om |u_n||\nabla u_n|\,dx+\int_{\partial^*\Om\cup \Gamma} |\gamma_l(u_n^2)-\gamma_r(u_n^2)|\,d\hn\right)\\
\le \frac{1}{2}\|u_n\|^2_{L^2(\Om)}+C\left(\int_\Om |\nabla u_n|^2\,dx+\int_{\partial^*\Om\cup \Gamma}[\gamma_l(u_n)^2+\gamma_r(u_n)^2]\,d\hn\right)\\
\le 
 \frac{1}{2}\|u_n\|^2_{L^2(\Om)}+C\left(\int_\Om |\nabla u|^2\,dx+\int_{\partial^*\Om\cup \Gamma}[\gamma^2_l(u)+\gamma^2_r(u)]\,d\hn\right),
\end{multline*}
where $C$ depends on $|\Om|$ but not on $n$. The conclusion follows by letting $n\to+\infty$.
\end{proof}

\begin{remark}
\label{rem:2dd-1}
{\rm
The computations of the previous proof shows that indeed $u\in L^{\frac{2N}{N-1}}(\Om)$ with
\begin{equation}
\label{eq:est2dd-1}
\|u\|_{L^{\frac{2N}{N-1}}(\Om)}\le C\left( \|\nabla u\|_{L^2(\Om;\R^N)}+\|\gamma_l(u)\|_{L^2(\partial^*\Om\cup \Gamma)}+\|\gamma_r(u)\|_{L^2(\partial^*\Om\cup \Gamma)} \right).
\end{equation}
}
\end{remark}

By generalizing to $\as(\R^N)$ the approach developed in \cite{bgt2} concerning open bounded sets with rectifiable topological boundary in dimension two, we can endow the space $\Theta(\Omega,\Gamma)$ with a scalar product which turns it into a Hilbert space, a variant of which can be used to generalize the Robin boundary value problem.
For every $u,v\in \Theta(\Omega,\Gamma)$ let us set
\begin{equation}
\label{eq:scalarpd}
(u,v)_{\Theta(\Omega,\Gamma)}:=\int_\Omega \nabla u\cdot\nabla v\,dx
+\int_{\partial^*\Omega\cup\Gamma}\left(\gamma_l(u)\gamma_l(v)+\gamma_r(u)\gamma_r(v)\right)\:d\hn.
\end{equation}
The following property holds true.

\begin{proposition}[\bf Hilbert space structure]
\label{prop:hilbert}
Let $(\Omega,\Gamma)\in\mathcal{A}(\R^N)$. The space $\Theta(\Omega,\Gamma)$ is a Hilbert space with respect to the scalar product \eqref{eq:scalarpd}. Moreover
the embedding $\Theta(\Omega,\Gamma)\hookrightarrow L^2(\Om)$ is compact.
\end{proposition}

\begin{proof}
In view of Lemma \ref{lem:L2}, we need simply to check the completeness of the scalar product. Let $(u_n)_{n\in\N}$ be a Cauchy sequence in $\Theta(\Omega,\Gamma)$ with respect to the scalar product \eqref{eq:scalarpd}. Taking into account Lemma \ref{lem:L2}, there exist $\Phi\in L^2(\R^N;\R^N)$, $u\in L^2(\R^N)$, $\alpha_1,\alpha_2\in L^2(\partial^*\Om\cup\Gamma)$ such that $\Phi=u=0$ a.e. on $\Om^c$, 
$$
u_n\to u\qquad\text{strongly in }L^2(\R^N),
$$
$$
\nabla u_n\to \Phi\qquad\text{strongly in }L^2(\R^N;\R^N),
$$
$$
\gamma_{l/r}(u_n)\to \alpha_{l/r} \qquad\text{strongly in }L^2(\partial^*\Om\cup\Gamma).
$$
It is readily seen, thanks to Ambrosio's theorem (see \cite[Theorem 4.7]{AFP}), that $u\in SBV(\R^N)$ with $\Phi=\nabla u$ and $J_u\tsub \partial^*\Om\cup\Gamma$. Moreover, we have that $(u_n)_{n\in\N}$ is a Cauchy sequence in $BV(\R^N)$: indeed thanks to Lemma \ref{lem:L2} we may write 
\begin{multline*}
\|u_n-u_m\|_{BV(\R^N)}\\
\le \|\nabla u_n-\nabla u_m\|_{L^1(\R^N;\R^N)}
+\int_{\partial^*\Om\cup \Gamma}\left[|\gamma_l(u_n-u_m)|+|\gamma_r(u_n-u_m)|\right]\,d\hn +\|u_n-u_m\|_{L^1(\Om)} \\
\le C\left[ \|\nabla u_n-\nabla u_m\|_{L^2(\R^N;\R^N)}+\|\gamma_l(u_n-u_m)\|_{L^2(\partial^*\Om\cup\Gamma)}\right.\\
\left.+\|\gamma_r(u_n-u_m)\|_{L^2(\partial^*\Om\cup\Gamma}\right]^{1/2}\to 0.
\end{multline*}
Thanks to the continuity of the (locally defined) trace operators $\gamma_{l/r}$ with respect to the BV norm, we infer that 
$$
\alpha_{l/r}=\gamma_{l/r}(u).
$$
We conclude that $u\in \Theta(\Om,\Gamma)$ and that $u_n\to u$ with respect to the scalar product \eqref{eq:scalarpd}.
\par
Let us check the compact embedding into $L^2(\Om)$. Let $(u_n)_{n\in\N}$ be a bounded sequence in $\Theta(\Om,\Gamma)$. In view of 
\eqref{eq:est2dd-1}, we have that $(u_n)_{n\in\N}$ is bounded in $L^{\frac{2N}{N-1}}(\Om)$. Since $(u_n)_{n\in\N}$ is bounded also in $BV(\R^N)$, and supported in $\Om$ which has finite measure, up to a subsequence we have that $u_n$ converges to some $u$ strongly in $L^1(\Om)$. From the previous bounds, we have that $u\in L^{\frac{2N}{N-1}}(\Om)$ and by interpolation that the convergence is strong in $L^2(\Om)$, so that the conclusion follows.
\end{proof}

\begin{remark}[\bf Independence on the orientation]
\label{rem:ind}
{\rm
Notice that the space $\Theta(\Om,\Gamma)$ and its associated Hilbert structure given by the scalar product \eqref{eq:scalarpd} are intrinsically defined, i.e., they do not depend on the decomposition and on the orientation of $\partial^*\Om\cup \Gamma$ used to define the trace operators $\gamma_l,\gamma_r$ according to Remark \ref{rem:traces}. Indeed, in view of the general theory for $BV$ functions (see \cite[Theorem 3.77]{AFP}), traces admit an a.e. pointwise characterization as Lebesgue values on semiballs (the direction being that of the chosen normal): as a consequence, every choice for the decomposition or the orientation will cause simply a possible switch between $\gamma_l(u)$ and $\gamma_r(u)$ (so that the condition defining $\Theta(\Om,\Gamma)$ remains the same), but will pair $\gamma_{l/r}(u)$ with the corresponding $\gamma_{l/r}(v)$, leaving the scalar product unchanged.
}
\end{remark}

\subsection{Generalization of the Robin-Laplacian boundary value problem}
\label{sec:gen-rob}
Let us fix $\beta>0$ and let $(\Om,\Gamma)\in \as(\R^N)$. Given $f\in L^2(\Om)$ we say that $u\in \Theta(\Om,\Gamma)$ is the solution of the Robin-Laplacian boundary value problem on $(\Om,\Gamma)$ with coefficient $\beta$ 
$$
\begin{cases}
-\Delta u=f&\text{in }\Om\setminus \Gamma\\
\frac{\partial u}{\partial n}+\beta u=0&\text{on }\partial^*\Om\cup \Gamma
\end{cases}
$$
if
\begin{equation}
\label{eq:rob}
\forall v\in \Theta(\Om,\Gamma)\,:\, a_\beta(u,v)=\int_\Om fv\,dx,
\end{equation}
where $a_\beta(u,v)$ is the bilinear form
$$
a_\beta(u,v):=(\nabla u,\nabla v)_{L^2(\Om;\R^N)}+\beta (\gamma_l(u),\gamma_l(v))_{L^2(\partial^*\Om\cup \Gamma)}+\beta (\gamma_r(u),\gamma_r(v))_{L^2(\partial^*\Om\cup \Gamma)}.
$$
Thanks to the results of Section \ref{sec:geom}, the following proposition holds true.

\begin{proposition}[\bf Existence and uniqueness of a solution]
\label{prop:exist}
Let $(\Om,\Gamma)\in \as(\R^N)$. Then for every $f\in L^2(\Om)$ there exists one and only one solution in $\Theta(\Om,\Gamma)$ to problem \eqref{eq:rob}. Moreover the resolvent operator from $L^2(\Om)$ into $\Theta(\Om,\Gamma)$ is compact.
\end{proposition}

\begin{proof}
Existence and uniqueness follow by Lax-Milgram theorem. Concerning the compactness of the resolvent operator, let $(f_n)_{n\in\N}$ be a bounded sequence in $L^2(\Om)$, and let $u_n\in\Theta(\Om,\Gamma)$ be the solution associated to $f_n$. Up to a subsequence we may assume 
$$
f_n\weak f\qquad\text{weakly in }L^2(\Om)
$$
and
$$
u_n\weak u\qquad\text{weakly in }\Theta(\Om,\Gamma).
$$
We immediately infer that $u$ is the solution associated to the right hand side $f$.
Moreover, thanks to the compact embedding in $L^2(\Om)$ given by Proposition \ref{prop:hilbert} we have
$$
u_n\to u\qquad\text{strongly in }L^2(\Om).
$$
We can thus write
$$
a_\beta(u_n,u_n)=\int_\Om f_n u_n\,dx \to \int_\Om fu\,dx=a_\beta(u,u)
$$
which yields the strong convergence in $\Theta(\Om,\Gamma)$ of $u_n$ to $u$, so that the proof is concluded.
\end{proof}

\subsection{Eigenvalues of the generalized Robin-Laplacian}
\label{sec:gen-lbk}
Using the weak formulation of the Robin-Laplacian, we can immediately define eigenvalues and eigenfunctions: we say that $\tilde\lambda$ is an eigenvalue with associated eigenfunction $u\in \Theta(\Om,\Gamma)$ if $u\not=0$ and
$$
\forall v\in \Theta(\Om,\Gamma)\,:\, a_\beta(u,v)=\tilde\lambda\int_\Om uv\,dx.
$$
In view of the Hilbert space structure of the generalized boundary value problem \eqref{eq:rob}, and of the properties of the associated resolvent operator (see Proposition \ref{prop:exist}), we deduce the existence of a sequence of eigenvalues
$\lkb(\Om,\Gamma)\to +\infty$ which can be characterized (taking multiplicity into account) through the min-max formula
\begin{equation}
\label{eq:lambdak}
\lkb(\Omega,\Gamma)=\min_{\overset{V\subset\Theta(\Omega,\Gamma)}{\dim V=k}}\max_{v\in V\setminus\left\{0\right\}}R_\beta(v),
\end{equation}

where $R_\beta(v)$ is the Rayleigh quotient
\begin{equation}
\label{eq:Rbeta}
R_\beta(v):=
\frac{\displaystyle\int_{\Om}|\nabla v|^2\:dx+\beta\int_{\partial^*\Omega\cup\Gamma}[\gamma_l^2(v)+\gamma_r^2(v)]\:d\hn}{\displaystyle\int_{\Om}v^2\:dx}.
\end{equation}
It turns out that the space which provides the minimum in \eqref{eq:lambdak} is given by $V_k:=\text{span}\{u_1,\dots,u_k\}$, where for $j\le k$ the function $u_j$ is a $L^2$-normalized eigenfunction associated to $\tilde{\lambda}_{j,\beta}(\Om,\Gamma)$.

\begin{remark}
{\rm
Notice that, if $\Omega\subset\R^N$ is a Lipschitz domain 
$$
\lambda_{k,\beta}(\Omega)=\lkb(\Omega,\emptyset).
$$
Moreover if $u_1,\ldots,u_k\in W^{1,2}(\Omega)$ are the first $k$ Robin eigenfunctions on $\Omega$, then $u_1,\ldots, u_k$, extended by zero outside $\Omega$, belong to $\Theta(\Omega,\emptyset)$ and provide the first $k$-eigenfunction of the generalized formulation.
}
\end{remark}

\begin{remark}[\bf Faber-Krahn inequality]
\label{rem:faber}
{\rm
For every $(\Om,\Gamma)\in \as(\R^N)$, the following generalized Faber-Krahn inequality holds true:
\begin{equation}
\label{eq:faber}
\lambda_{1,\beta}(B)\le \tilde \lambda_{1,\beta}(\Om,\Gamma),
\end{equation}
where $B$ is a ball such that $|B|=|\Om|$. In other words, balls are optimal domains for $\tilde \lambda_{1,\beta}$ under a volume constraint. Indeed if $u\in \Theta(\Om,\Gamma)$ is the first eigenfunction of $(\Om,\Gamma)$, then we can assume $u\ge 0$ so that  $u\in SBV^{1/2}(\R^N)$ according to the notation introduced in \cite{BucGiac} (i.e., $u^2\in SBV(\R^N)$). Then
\begin{multline*}
\tilde \lambda_{1,\beta}(\Om,\Gamma)=R_\beta(u)= \frac{\int_{\Om}|\nabla u|^2\,dx+\beta \int_{\partial^*\Om\cup \Gamma}[\gamma_l^2(u)+\gamma_r^2(u)]\,d\hn}{\int_{\Om}u^2\,dx} \\
\ge \frac{\int_{\R^N}|\nabla u|^2\,dx+\beta \int_{J_{u}}[\gamma_l^2(u)+\gamma_r^2(u)]\,d\hn}{\int_{\R^N}u^2\,dx}\ge \lambda_{1,\beta}(B),
\end{multline*}
the last inequality following from \cite[Theorem 3]{BucGiac}, so that \eqref{eq:faber} holds true.
}
\end{remark}

\begin{remark}[\bf Scaling property and monotonicity under dilations]
\label{rem:scaling}
{\rm
The standard scaling property for $\lambda_{k,\beta}$ extends readily to $\lkb$: for $t>0$ we have
$$
\lkb(t\Omega,t\Gamma)=\frac{1}{t^2}\tilde{\lambda}_{k,t\beta}(\Omega,\Gamma).
$$
In particular, we have that 
$$
\lkb(t\Omega,t\Gamma)<\tilde{\lambda}_{k,\beta}(\Omega,\Gamma)
$$ 
for every $t>1$: we will refer to this inequality as the monotonicity under dilation of $\lkb$.
}
\end{remark}

\begin{remark}[\bf Spectrum of a disjoint union]
\label{rem:disjoint}
{\rm
Let $(\Om_1,\Gamma_1), (\Om_2,\Gamma_2)\in \as(\R^N)$ with $\Omega_1^1\cap \Omega_2^1=\emptyset$ and $\Omega_1^1,\Omega_2^1$ lying at positive distance. Then, using the weak equation for the eigenfunctions, it is readily seen that the spectrum of $(\Omega_1\cup\Omega_2,\Gamma_1\cup\Gamma_2)$ is the union of the spectra of $(\Omega_1,\Gamma_1)$ and $(\Omega_2,\Gamma_2)$. In particular we have the standard formula
$$
\lkb(\Omega_1\cup\Omega_2,\Gamma_1\cup\Gamma_2)=\min_{i=0,\ldots,k}\max\left\{\tilde\lambda_{i,\beta}(\Omega_1,\Gamma_1),\tilde\lambda_{k-i,\beta}(\Omega_2,\Gamma_2)\right\},
$$
where we assume $\tilde{\lambda}_{0,\beta}(\Omega_j,\Gamma_j):=0$.
}
\end{remark}

The abstract Hilbert space formulation of the problem yields the generalization to the present setting of the classical result on the boundedness of eigenfunctions. 

\begin{theorem}[\bf Boundedness of the eigenfunctions]
\label{th:boundeigen}
Let $(\Om,\Gamma)\in \as(\R^N)$, and let $\tilde\lambda>0$ be an eigenvalue for \eqref{eq:rob}. Then, there exists a positive constant $C
$, depending only on $N$, $|\Om|$ and $\beta$, such that for every ($L^2$-normalized) eigenfunction $u$ for $\tilde\lambda$ it holds
\begin{equation}
\label{eq:uinfty}
\|u\|_\infty\le C \tilde\lambda^N.
\end{equation}
\end{theorem}
\begin{proof}
Let $u$ be an eigenfunction for $\tilde\lambda$, i.e., $u\not=0$ and
\begin{equation}
\label{eq:dual2}
\forall v\in \Theta(\Om,\Gamma)\,:\, a_\beta(u,v)=\tilde\lambda (u,v)_{L^2(\Om)}.
\end{equation}
For every $0<h<\|u\|_\infty$, the function
$$u_h:=(u-h)_+-(u+h)_-\in\Theta(\Omega,\Gamma)$$
is not identically null. Using $u_h$ as a test in equation \eqref{eq:dual2} we get
$$\int_{\Omega}\nabla u\cdot\nabla u_h\:dx+\beta\int_{\partial^*\Omega\cup\Gamma}[\gamma_l(u)\gamma_l(u_h)+\gamma_r(u)\gamma_r(u_h)]\:d\hn=\tilde\lambda\int_\Omega u\, u_h\:dx.$$
Since $u_h$ is supported in $A_h:=\{|u|>h\}$ and $\nabla u_h=\nabla u$ on $A_h$ one gets
\begin{align*}
\int_{\Omega}|\nabla u_h|^2\:dx&+\beta\int_{(\partial^*\Omega\cup\Gamma)\cap\{u>h\}}[\gamma_l(u)\gamma_l(u-h)+\gamma_r(u)\gamma_r(u-h)]\:d\hn\\
&+\beta\int_{(\partial^*\Omega\cup\Gamma)\cap\{u<-h\}}[\gamma_l(u)\gamma_l(u+h)+\gamma_r(u)\gamma_r(u+h)]\:d\hn\\
&=\tilde\lambda\int_{\{u>h\}} u(u-h)\:dx+\tilde\lambda\int_{\{u<-h\}} u(u+h)\:dx.
\end{align*}
Consequently, we obtain
\begin{equation}\label{eq:lim1}
\begin{split}
\int_{\Omega}|\nabla u_h|^2\:dx&+\beta\int_{\partial^*\Omega\cup\Gamma}[\gamma_l^2(u_h)+\gamma_r^2(u_h)]\:d\hn\\
&\le\tilde\lambda\int_{\Omega} u_h^2\:dx+\tilde\lambda h\int_{\Omega} |u_h|\:dx\\
&\le\tilde\lambda\int_{\Omega} u_h^2\:dx+\frac{\tilde\lambda}{2}\left[\int_{\Omega}u_h^2\:dx+h^2|A_h|\right]\\
&=\frac{3\tilde\lambda}{2}\int_{\Omega}u_h^2\:dx+h^2\frac{\tilde\lambda}{2}|A_h|.
\end{split}
\end{equation}
In view of H\"older inequality and of Remark \ref{rem:2dd-1} we get
\begin{equation}\label{eq:lim2}
\begin{split}
\|u_h\|_{L^2(\Omega)}&\le \|u_h\|_{L^\frac{2N}{N-1}(\Omega)} |A_h|^{\frac{1}{2N}}\\
&\le  C_1\left[\int_{\Omega}|\nabla u_h|^2\:dx+\beta\int_{\partial^*\Omega\cup\Gamma}[\gamma_l^2(u_h)+\gamma_r^2(u_h)]\:d\hn\right]^{\frac{1}{2}}|A_h|^{\frac{1}{2N}},
\end{split}
\end{equation}
so that plugging the estimate in \eqref{eq:lim1},
\begin{multline}
\label{eq:lim3}
\int_{\Omega}|\nabla u_h|^2\:dx+\beta\int_{\partial^*\Omega\cup\Gamma}[\gamma_l(u_h)^2+\gamma_r^2(u_h)]\:d\hn\\
\le \frac{3\tilde\lambda}{2} C_1^2\left[\int_{\Omega}|\nabla u_h|^2\:dx+\beta\int_{\partial^*\Omega\cup\Gamma}[\gamma_l(u_h)^2+\gamma_r^2(u_h)]\:d\hn\right]|A_h|^{\frac{1}{N}}+h^2\frac{\tilde\lambda}{2}|A_h|.
\end{multline}
Here $C_1$ depends on $N$, $|\Om|$ and $\beta$.
\par

Let us consider $h_0>0$ such that 
$$
\frac{3\tilde\lambda}{2}C_1^2 \left(\frac{1}{h_0}\int_{\Om}|u|\,dx\right)^{\frac{1}{N}}=\frac{1}{2}.
$$
If $\|u\|_\infty \le h_0$, then \eqref{eq:uinfty} immediately follows. Let us thus assume that $\|u\|_\infty>h_0$. Then for $h_0\le h<\|u\|_\infty$ we have
$$
|A_h|\le \frac{1}{h}\int_\Om |u|\,dx
$$
so that in view of the choice of $h_0$ the first term in the right hand side of \eqref{eq:lim3} can be absorbed by the left hand side leading to 
\begin{equation}\label{eq:lim4}
\int_{\Omega}|\nabla u_h|^2\:dx+\beta\int_{\partial^*\Omega\cup\Gamma}[\gamma_l(u_h)^2+\gamma_r(u_h)^2]\:d\hn\le h^2\tilde\lambda|A_h|.
\end{equation}
In view of \eqref{eq:lim2} we get
$$
\|u_h\|_{L^2(\Omega)}\le C_1 h\sqrt{\tilde\lambda}|A_h|^{\frac 12+\frac{1}{2N}},
$$
so that by applying H\"older inequality we deduce the key estimate
\begin{equation}\label{eq:lim5}
\|u_h\|_{L^1(\Omega)}\le C_1 h\sqrt{\tilde\lambda}|A_h|^{1+\frac{1}{2N}}.
\end{equation}
By setting
$$
g(h):=\int_\Omega|u_h|\:dx=\|u_h\|_{L^1(\Omega)},
$$
one has $g'(h)=-|A_h|$ for a.e. $h>0$ and \eqref{eq:lim5} becomes
$$
g(h)\le C_1 h\sqrt{\tilde\lambda}(-g'(h))^{1+\frac{1}{2N}},
$$
so that
$$
\frac{1}{h^{1-\frac{1}{2N+1}}}\le- (C_1\sqrt{\tilde\lambda})^{\frac{2N}{2N+1}} \frac{g'(h)}{g(h)^{1-\frac{1}{2N+1}}}.
$$
This last inequality shows that $\|u\|_\infty<+\infty$. Indeed, if this is not the case, integrating from $h_0$ to $+\infty$, we have that the left hand side diverges, but the right hand side converges.
\par
Integrating from $h_0$ to $\|u\|_\infty$ we deduce 
$$
\|u\|_\infty^{\frac{1}{2N+1}}-h_0^{\frac{1}{2N+1}}\le (C_1\sqrt{\tilde\lambda})^{\frac{2N}{2N+1}} \|u\|_1^{\frac{1}{2N+1}},
$$
from which inequality \eqref{eq:uinfty} follows. The proof is thus concluded.
\end{proof}

\section{The main results}
\label{sec:main}
Being interested in a perimeter constraint for our optimization problem, we need to generalize the notion of perimeter to admissible configurations in $\as(\R^N)$.

\begin{definition}[\bf Generalized perimeter]
\label{def:robinperimeter}
For every $(\Omega,\Gamma)\in \as(\R^N)$ we set
$$
\Prob{\Omega}{\Gamma}:=Per(\Omega)+2\hn(\Gamma),
$$
where $Per(\Om)$ denotes the usual perimeter of $\Om$ in $\R^N$.
\end{definition}

\begin{remark}
{\rm
The definition of the generalized perimeter for the configuration $(\Om,\Gamma)$ is based on the idea that the inner crack $\Gamma$ is seen as a degenerated hole or a degenerated inner fold of the outer boundary, so that its contribution to the perimeter is given by twice its surface $\hn(\Gamma)$.
}
\end{remark}

The first main result of the paper is the following.

\begin{theorem}
\label{th:main1}
Let $p>0$. The problem
\begin{equation}\label{eq:jumprobinconst}
\min\left\{\lkb(\Omega,\Gamma): (\Omega,\Gamma)\in\mathcal{A}(\R^N) \text{ with }\Prob{\Omega}{\Gamma}=p\right\}
\end{equation}
is well posed. Moreover every minimizer $(\Omega,\Gamma)$ is such that $\Omega$ is bounded and
\begin{equation}
\label{eq:den1}
\lkb(\Omega,\Gamma)=\inf \{\lambda_{k,\beta}(A)\,:\, \text{$A\subset \R^N$ is a bounded Lipschitz open set with $Per(A)=p$}\}.
\end{equation}
\end{theorem}

\begin{remark}[\bf The case $k=1$]
\label{rem:1ball}
{\rm
For $k=1$, thanks to the Faber-Krahn type inequality \eqref{eq:faber} and to the monotonicity of the eigenvalues under dilations, it is readily seen that the only minimizers of both problems are the couples $(B,\emptyset)$, where $B$ are balls.
}
\end{remark}

Thanks to Theorem \ref{th:main1}, we can handle more general problems involving several eigenvalues at the same time. 
Let $\ell\in\N$, $\ell\ge 1$ and $f:]0,+\infty[^\ell \to ]0,+\infty[ $ be a Lipschitz continuous function such that the following items hold true.
\begin{itemize}
\item[$(f1)$] $f(y)\to+\infty$ as $|y|\to+\infty$;
\item[$(f2)$] There exists $C>0$ such that for every $y=(y_1,\ldots, y_\ell),y'=(y'_1,\ldots, y'_\ell)\in ]0,+\infty[^{\ell}$ with $y_i\ge y'_i$ for every $i=1,\ldots,\ell$ it holds
$$
f(y)\ge f(y')+C|y-y'|.
$$
\end{itemize}
A typical example for $f$ is given by $f(y):=y_1+\dots+y_\ell$ or more generally $f(y):=(y_1^p+\dots+y_\ell^p)^{\frac{1}{p}}$ with $p>1$.
\par 
Given $k_1,\ldots,k_l\in\N\setminus\{0\}$ let us set
\begin{equation*}
F(\Omega,\Gamma):=f\left(\tilde\lambda_{k_1,\beta}(\Omega,\Gamma),\ldots,\tilde\lambda_{k_l,\beta}(\Omega,\Gamma)\right).
\end{equation*}
The following result holds true.

\begin{theorem}
\label{thm:main3}
Let $p>0$, and let $f:]0,+\infty[^\ell \to ]0,+\infty[$ satisfy $(f1)$ and $(f2)$.  Then the problem
\begin{equation}
\label{eq:pb-gen}
\min\left\{F(\Omega,\Gamma): (\Omega,\Gamma)\in\mathcal{A}(\R^N) \text{ with }\Prob{\Omega}{\Gamma}=p\right\}
\end{equation}
is well posed. Moreover every minimizer $(\Omega,\Gamma)$ is such that $\Omega$ is bounded
and\begin{equation}
\label{eq:den-gen}
F(\Omega,\Gamma)=\inf \{F(A,\emptyset)\,:\, \text{$A\subset \R^N$ is a bounded Lipschitz open set with $Per(A)=p$}\}.
\end{equation}
\end{theorem}

\begin{remark}[\bf Penalization of the perimeter]
\label{rem:penal}
{\rm
We can also treat the problem where the perimeter is involved as a penalization term, i.e.,
\begin{equation}\label{eq:jumprobinpenal-gen}
\min\left\{F(\Om,\Gamma)+\Lambda \Prob{\Omega}{\Gamma}:(\Omega,\Gamma)\in\mathcal{A}(\R^N)\right\},
\end{equation}
where $\Lambda>0$. Also problem \eqref{eq:jumprobinpenal-gen} is well posed, every minimizer is bounded, and the minimum value is given by the infimum of the values achieved on Lipschitz regular sets (see Remark \ref{rem:pen-proof}). 
}
\end{remark}

\begin{remark}
\label{rem:beta<0}
{\rm
In the case $\beta<0$, eigenvalues for regular domains are still well defined, and the interesting associated variational problem is their maximization, both for the perimeter and the measure constraint. As highlighted in \cite{Bucur-Cito,bucur2019sharp}, the problem can be framed within the class of measurable sets of finite perimeter, without taking into account inner cracks. The reason of that is the negative sign of the surface energies: roughly speaking, one can take the eigenfunctions for a couple $(\Omega,\emptyset)$ as test functions for $(\Omega,\Gamma)$, obtaining a lower value of the functional and so a worse competitor for the associated maximization. 
}
\end{remark}

%

\section{Compactness, lower semicontinuity and approximation results}
\label{sec:tech}
In this section we collect some technical results which are fundamental for our analysis: in particular we are interested in compactness properties of admissible configurations, lower semicontinuity results for the associated eigenvalues, and their approximation through regular sets.

\subsection{Compactness and lower semicontinuity properties}
The following result holds true.

\begin{theorem}[\bf Compactness and lower semicontinuity]
\label{thm:lsc}
Let $(\Omega_n,\Gamma_n)\in\mathcal{A}(\R^N)$ be such that 
\begin{equation}
\label{eq:boundProb}
\Prob{\Omega_n}{\Gamma_n}\le C 
\end{equation}
for some positive constant $C>0$ independent of $n$. Assume that 
\begin{equation}
\label{eq:s-omega_n}
1_{\Om_n}\to 1_\Om\qquad\text{strongly in }L^1(\R^N)
\end{equation}
for some set of finite perimeter $\Om\subset \R^N$. 
\par
Then there exists $\Gamma\subset \R^N$ rectifiable such that $(\Omega,\Gamma)\in\mathcal{A}(\R^N)$ and the following items hold true.
\begin{itemize}
\item[(i)] 
We have
\begin{equation}
\label{eq:lscboundary}
\Prob{\Omega}{\Gamma}\le\liminf_{n\to+\infty}\Prob{\Omega_n}{\Gamma_n}.
\end{equation}

\item[(ii)] If $u_n\in \Theta(\Om_n,\Gamma_n)$  with
\begin{equation}
\label{eq:bound-rob}
\int_{\Omega_n}|\nabla u_n|^2\:dx+\int_{\partial^*\Om\cup \Gamma_n}[\gamma_l^2(u_n)+\gamma_r^2(u_n)]\:d\hn\le C,
\end{equation}
then there exists $u\in \Theta(\Om,\Gamma)$ such that, up to subsequences,
$$
u_n\to  u\quad\text{strongly in $L^2(\R^N)$},
$$
$$
\nabla u_n\weak \nabla u\quad\text{weakly in $L^2(\R^N;\R^N)$},
$$
and
\begin{equation}
\label{eq:superficie}
\int_{\partial^*\Omega\cup\Gamma}[\gamma_l^2(u)+\gamma_r^2(u)]\:d\hn\le\liminf_{n\to+\infty}\int_{\partial^*\Omega_n\cup\Gamma_n}[\gamma_l^2(u_n)+\gamma_r^2(u_n)]\:d\hn.
\end{equation}
\end{itemize}
\end{theorem}

\begin{proof}
We will make use of the notion of $\sigma^2$-convergence of rectifiable sets introduced in \cite{DMFT} (see Section \ref{sec:sigma2}).
\par
From
$$
Per(\Om_n)+2\hn(\Gamma_n)=\Prob{\Om_n}{\Gamma_n}\le C,
$$
and Theorem \ref{thm:sigma2}, employing also a diagonal argument, we deduce that there exists a rectifiable set $K\subset \R^N$  such that, up to a subsequence (not relabelled), for every $D\subset \R^N$ open and bounded we have
$$
(\partial^*\Omega_n\cup\Gamma_n)\cap D \to K\cap D \qquad\text{in the sense of $\sigma^2$-convergence}.
$$
Since $J_{1_{\Omega_n \cap D}} \tsub (\partial^*\Omega_n\cup\Gamma_n)\cap D$ and $\nabla 1_{\Omega_n \cap D}=\nabla 1_{\Om\cap D}=0$, in view of \eqref{eq:s-omega_n} and of property (a) in Definition \ref{def:sigma2} of $\sigma^2$-convergence, we deduce that 
$$
\partial^*\Omega \cap D\tsub K\cap D.
$$
Since $D$ is arbitrary, we infer $\partial^*\Om\tsub K$. Let now decompose $K$ as 
$$
K=(K\cap\Omega^0)\cup\partial^*\Omega\cup(K\cap\Omega^1)
$$
and let us set
$$
\Gamma:=K\cap\Omega^1,
$$
so that $(\Om,\Gamma)\in \as(\R^N)$. 
\par
We divide the proof in several steps.

\vskip10pt\noindent{\bf Step 1.}
We claim that for every $D\subset \R^N$ open and bounded, there exist $v,v_n\in SBV(D)$ with $\|v\|_\infty,\|v_n\|_\infty\le C$, $v=v_n=0$ a.e. outside $\Om$ and $\Om_n$ respectively, $\nabla v,\nabla v_n\in L^2(D;\R^N)$, $J_v \teq (\partial^*\Omega\cup\Gamma)\cap D$, $J_{v_n}\tsub (\partial^*\Omega_n\cup\Gamma_n)\cap D$, such that
$$
v_n \to v\quad\text{strongly in $L^2(D)$},
$$
and
$$
\nabla v_n \weak \nabla v\quad\text{weakly in $L^2(D;\R^N)$}.
$$
By definition of $\sigma^2$-convergence there exist $w,w_n\in SBV(D)$ with $\|w\|_\infty,\|w_n\|_{\infty}\le C$, $J_w \teq K\cap D$, $J_{w_n}\tsub (\partial^*\Omega_n\cup\Gamma_n)\cap D$, such that
$$
w_n\to  w \qquad\text{strongly in }L^1(D)
$$
and
$$
\nabla w_n\weak  \nabla w \qquad\text{weakly in }L^2(D;\R^N).
$$
The first convergence is indeed also a convergence in $L^2(D)$ in view of the uniform bound on the $L^\infty$-norms. For $\varepsilon>0$ let us set
$$
v:=(w+\e)1_{\Om\cap D}\qquad\text{and}\qquad v_n:=(w_n+\e)1_{\Omega_n\cap D}.
$$
Clearly we have
$$
\Gamma \cap D\tsub J_{w1_{\Omega\cap D}} \tsub (\partial^*\Omega\cup\Gamma)\cap D
$$
and so, for a.e. $\varepsilon>0$ we deduce (see Remark \ref{rem:jumpsum})
$$
J_v \teq J_{w1_{\Omega\cap D}}\cup J_{1_{\Omega\cap D}} \teq (\partial^*\Omega\cup\Gamma)\cap D.
$$
Since
$$
J_{v_n} \tsub (\partial^*\Omega_n\cup\Gamma_n)\cap D,
$$
$v_n\to v$ strongly in $L^2(D)$ and $\nabla v_n\weak \nabla v$ weakly in $L^2(D;\R^N)$, we get that the claim follows by choosing $\e$ outside a negligible set.

\vskip10pt\noindent{\bf Step 2.}
Let us check item (i). Let $D\subset \R^N$ be open and bounded and let $\varepsilon>0$. Let us consider the functions on $D$
$$
\psi_\varepsilon:=1_{\Omega}+\varepsilon v \qquad\text{and}\qquad
\psi_{n,\varepsilon}:=1_{\Omega_n}+\varepsilon v_n,
$$
where $v,v_n$ are the functions given by Step 1. For a.e. $\e>0$ we have that 
$$
J_{\psi_\varepsilon}\teq J_{1_\Om}  \cup J_{v}  \teq(\partial^*\Omega\cup\Gamma)\cap D
\qquad\text{and}\qquad
J_{\psi_{n,\varepsilon}}\teq J_{1_{\Om_n}}  \cup J_{v_n} \tsub (\partial^*\Omega_n\cup\Gamma_n)\cap D.
$$
Since
$$
\psi_{n,\varepsilon}\to \psi_\varepsilon\qquad\text{strongly in }L^1(D)
$$
and
$$
\nabla \psi_{n,\varepsilon}\weak \nabla \psi_\varepsilon\qquad\text{weakly in }L^2(D;\R^N),
$$
lower semicontinuity in $SBV$ entails
$$
\int_{J_{\psi_\varepsilon}}[\gamma_l^2(\psi_\varepsilon)+\gamma_r^2(\psi_\varepsilon)]\:d\hn
\le\liminf_{n\to+\infty}\int_{J_{\psi_{n,\varepsilon}}}[\gamma_l^2(\psi_{n,\varepsilon})+\gamma_r^2(\psi_{n,\varepsilon})]\:d\hn.
$$
Taking into account bound \eqref{eq:boundProb}, for a.e. $\e>0$ small enough we deduce
\begin{align*}
\hn(\partial^*\Om \cap D)+2\hn(\Gamma\cap D)-C\varepsilon&\le\int_{J_{\psi_\varepsilon}}[\gamma_l^2(\psi_\varepsilon)+\gamma_r^2(\psi_\varepsilon)]\:d\hn\\
&\le\liminf_{n\to+\infty}\int_{J_{\psi_{n,\varepsilon}}}[\gamma_l^2(\psi_{n,\varepsilon})+\gamma_r^2(\psi_{n,\varepsilon})]\:d\hn\\
& \le \liminf_{n\to+\infty} \left[\hn(\partial^*\Om_n \cap D)+2\hn(\Gamma_n\cap D)\right]+C\e \\
&\le\liminf_{n\to+\infty}\Prob{\Omega_n}{\Gamma_n}+C\varepsilon,
\end{align*}
so item (i) follows by letting $\varepsilon\to0^+$ and $D$ invade the whole $\R^N$.

\vskip10pt\noindent{\bf Step 3.} Let us come to item (ii).  Since
$$
\hn(J_{u_n}) \le \hn(\partial^*\Om_n \cup \Gamma_n)\le \Prob{\Om_n}{\Gamma_n}\le C,
$$
the uniform bound on the Robin energy \eqref{eq:bound-rob} and the convergence  \eqref{eq:s-omega_n} yield  
\begin{align*}
|Du_n|(\R^N)&=\int_{\Omega_n}|\nabla u_n|\:dx+\int_{J_{u_n}}\left|\gamma_l(u_n)-\gamma_r(u_n)\right|\:d\hn\\
&\le\|\nabla u_n\|_{L^2(\R^N)}| \Om_n |^{1/2}+\int_{J_{u_n}}[\left|\gamma_l(u_n)\right|+\left|\gamma_r(u_n)\right|]\:d\hn\\
&\le\|\nabla u_n\|_{L^2(\R^N)}| \Om_n |^{1/2}+\hn(J_{u_n})+\frac{1}{2}\int_{J_{u_n}}(\gamma_l(u_n)^2+\gamma_r(u_n)^2)\:d\hn\\
&\le  C_1
\end{align*}
for some $C_1>0$ independent of $n$. Taking into account Remark \ref{rem:2dd-1}, we deduce that there exists $u\in BV(\R^N)$ with 
\begin{equation}
\label{eq:u=0}
u=0 \quad \text{a.e. in $\Omega^c$}.
\end{equation}
and
$$
u_n\to u\qquad\text{strongly in }L^2(\R^N).
$$

Let us fix $D\subset \R^N$ open and bounded. In view of Ambrosio's theorem we infer that $u\in SBV(D)$ with
$$
\nabla u_n\weak \nabla u\qquad\text{weakly in }L^2(D;\R^N).
$$
Let us consider for every $M>0$ the truncated functions
$$
u_{n,M}:=\max\{\min\{u_n,M\},-M\}\qquad\text{and}\qquad u_M:=\max\{\min\{u,M\},-M\}\
$$
Clearly 
$$
u_{n,M}\to u_M\qquad\text{strongly in }L^2(D),
$$ 
and
$$
\nabla u_{n,M}\weak \nabla u_M\qquad\text{weakly in }L^2(D;\R^N).
$$ 
Since $J_{u_{n,M}} \tsub K_n\cap D$, from the properties of $\sigma^2$-convergence and since $u_M=0$ a.e. on $\Om^c\cap D$ we infer $J_{u_M} \tsub (\partial^*\Om \cup \Gamma)\cap D$. 
\par
In view of the arbitrariness of $M$ and $D$ we conclude that 
\begin{equation}
\label{eq:uSBV}
u\in SBV(\R^N)
\end{equation}
with
\begin{equation}
\label{eq:nablaunwD}
\nabla u_n\weak \nabla u\qquad\text{weakly in }L^2(\R^N;\R^N).
\end{equation}
and
\begin{equation}
\label{eq:JuGamma}
J_{u}\tsub\partial^*\Omega\cup\Gamma.
\end{equation}
Let us check \eqref{eq:superficie}. For every $\varepsilon>0$ let us set
$$
w_n:=u_n+\varepsilon v_n\qquad\text{and}\qquad w:=u+\varepsilon v,
$$ 
where $v_n,v$ are given by Step 1.  For a.e. $\varepsilon>0$ we have
$$
J_{w}=J_{u+\varepsilon v}\teq J_u\cup J_v\teq (\partial^*\Omega\cup\Gamma)\cap D
\qquad\text{and}\qquad
J_{w_n}=J_{u_n+\varepsilon v_n}\tsub(\partial^*\Omega_n\cup\Gamma_n)\cap D.
$$
By lower semicontinuity in $SBV$ we may write
\begin{align*}
\int_{(\partial^*\Omega\cup\Gamma)\cap D}[\gamma_l^2(u+\varepsilon v)&+\gamma_r^2(u+\varepsilon v)]\:d\hn=\int_{J_{w}}[\gamma_l^2(w)+\gamma_r^2(w)]\:d\hn\\
&\le\liminf_{n\to+\infty}\int_{J_{w_n}}[\gamma_l^2(w_n)+\gamma_r^2(w_n)]\:d\hn\\
&\le\liminf_{n\to+\infty}\int_{(\partial^*\Omega_n\cup\Gamma_n)\cap D}[\gamma_l^2(u_n+\varepsilon v_n)+\gamma_r^2(u_n+\varepsilon v_n)]\:d\hn.
\end{align*}
Since the functions $v,v_n$ are uniformly bounded in $L^\infty(D)$,  and in view of the perimeter bound \eqref{eq:boundProb}, we get by letting $\varepsilon\to0^+$, and since $D$ is arbitrary
\begin{equation}
\label{eq:robn}
\int_{\partial^*\Omega\cup\Gamma}[\gamma_l^2(u)+\gamma_r^2(u)]\:d\hn\le\liminf_{n\to+\infty}\int_{\partial^*\Omega_n\cup\Gamma_n}[\gamma_l^2(u_n)+\gamma_r^2(u_n)]\:d\hn\le C.
\end{equation}
Thanks to \eqref{eq:uSBV}, \eqref{eq:u=0}, \eqref{eq:JuGamma}, \eqref{eq:nablaunwD}, and \eqref{eq:robn}, we conclude that $u\in \Theta(\Om,\Gamma)$, so that item (ii) is proved.
\par
\end{proof}

\begin{remark}
\label{rem:jumpsum}
{\rm
In the previous proof we used the following result: if $D\subseteq \R^N$ is open and $u,v\in SBV(D)$ with $\hn(J_u),\hn(J_v)<+\infty$, then for a.e. $\e>0$
$J_{u+\e v}\teq J_u\cup J_v$. For a proof of this fact we refer to \cite[Remark 3.13]{BucGiac-lambdak}.
}
\end{remark}

A key result for our analysis is the following lower semicontinuity result for $\lkb$ along converging configurations according to Theorem \ref{thm:lsc}.

\begin{theorem}[\bf Lower semicontinuity of $\lkb$]
\label{pro:lscrkb}
Let $(\Omega_n,\Gamma_n)\in \as(\R^N)$ converge to $(\Omega,\Gamma)\in \as(\R^N)$ in the sense of Theorem \ref{thm:lsc}. Then
$$\lkb(\Omega,\Gamma)\le\liminf_{n\to+\infty}\lkb(\Omega_n,\Gamma_n).$$
\end{theorem}

\begin{proof}
Let us assume $\lkb(\Omega_n,\Gamma_n)\le C$ and let us consider $$V_n:=\text{span}\{u_{n,1},\ldots,u_{n,k}\}\subset\Theta(\Omega_n,\Gamma_n)$$ 
such that
$$\lkb(\Omega_n,\Gamma_n)=\max_{V_n}R_\beta.$$
Without loss of generality we can consider the $u_{n,j}$ to form a $L^2$-orthonormal family. Then we get for $j=1,\dots,k$
$$
\int_{\Om_n}|\nabla u_{n,j}|^2\,dx+\int_{\partial^*\Om_n\cup \Gamma_n}[\gamma_l^2(u_{n,j})+\gamma_r^2(u_{n,j})]\:d\hn\le \tilde C.
$$
By Theorem \ref{thm:lsc} we deduce that there exists $u_j\in \Theta(\Om,\Gamma)$ such that
$$
u_{n,j}\to u_j\quad\text{strongly in $L^2(\R^N)$}
$$
$$
\nabla u_{n,j}\weak\nabla  u_j\quad\text{weakly in $L^2(\R^N;\R^N)$}
$$
with
$$
\int_{\partial^*\Omega\cup\Gamma}\left[\gamma_l^2(u_j)+\gamma_r^2(u_j)\right]\:d\hn\\
\le\liminf_{n\to+\infty}\int_{\partial^*{\Omega_n}\cup{\Gamma_n}}\left[\gamma_l^2(u_{n,j})+\gamma_r^2(u_{n,j})\right]\:d\hn.
$$
Clearly $\{u_j\,:\, j=1,\dots,k\}$ forms an orthonormal family in $L^2$, so that $V:=\text{span}\{u_1,\ldots,u_k\}$ is a $k$-dimensional subspace of $\Theta(\Omega,\Gamma)$. Using again Theorem \ref{thm:lsc}, for every $a_j\in \R$ we have
$$
R_\beta\left(\sum_j a_j u_j\right) \le \liminf_n R_\beta\left(\sum_j a_j u_{n,j}\right) \le \liminf_n \max_{V_n} R_\beta=\liminf_n \lkb(\Om_n,\Gamma_n)
$$
which entails
$$
\lkb(\Omega,\Gamma)\le\max_{V} R_\beta\le \liminf_n \lkb(\Om_n,\Gamma_n),
$$
and the conclusion follows.
 
\end{proof}

\subsection{Approximation through regular domains}
The following result concerns the approximation of {\it bounded} configurations in $\as(\R^N)$ through regular domains.

\begin{theorem}
\label{thm:den}
Let $(\Omega,\Gamma)\in\mathcal{A}(\R^N)$ such that $\Om$ is bounded. Then, there exists a sequence of bounded regular open sets $(\Omega_n)_{n\in\N}$ such that
$$
\limsup_{n\to+\infty}\lambda_{k,\beta}(\Omega_n) \le \lkb(\Omega,\Gamma)
$$
and
$$
\limsup_{n\to+\infty}\hn(\partial\Omega_n)\le 
\Prob{\Omega}{\Gamma}.
$$
\end{theorem}

\begin{proof}
Let $u_1,\dots,u_k$ be the first $k$ eigenfunctions of $(\Om,\Gamma)$. It is not restrictive to assume $\Gamma=J_{(u_1,\dots,u_k)}$: indeed, by employing the min-max characterization of $\lkb$ we have
$$
\lkb(\Om, J_{(u_1,\dots,u_k)})\le \max_{\text{span}\{u_1,\dots,u_k\}} R_\beta=\lkb(\Om,\Gamma).
$$

We divide the proof in several steps.

\vskip10pt\noindent{\bf Step 1: Approximation from inside.} 
Let us approximate $\Omega$ with smooth sets ``from inside'' using Proposition \ref{pro:comitorres} with the choice $\mu:=\hn\lfloor (\partial^*\Om \cup \Gamma)$. We find $\Om_n\subset\R^N$ smooth open set such that
$$
|\Om_n \Delta\Omega|\to 0,\quad Per(\Om_n)\to Per(\Omega),
$$
with
$$
\hn(\partial^*\Om \cap \Om_n)\to 0, \qquad \hn(\Gamma\setminus \Om_n)\to 0, \qquad \hn(\partial \Om_n\setminus \Omega^1)\to 0.
$$
If $D$ is open and bounded such that $\Om\subset\subset D$, we can assume that also $\Om_n\subset\subset D$.
\par
If we set 
$$
\Gamma_n:= (\Gamma\cap \Om_n) \cup (\partial^*\Om \cap \Om_n),
$$
then clearly $(\Om_n,\Gamma_n)\in \as(\R^N)$ with
\begin{equation}
\label{eq:convprob1}
\lim_{n\to+\infty}\Prob{\Om_n}{\Gamma_n}=\Prob{\Om}{\Gamma}.
\end{equation}

We claim that
\begin{equation}
\label{eq:lambdaksup}
\limsup_{n\to+\infty} \lkb(\Om_n,\Gamma_n)\le \lambda_k(\Om,\Gamma).
\end{equation}
In order to prove this, we first show that $(\Om,\Gamma)$ is the limit configuration of the sequence $(\Om_n,\Gamma_n)$ according to Theorem \ref{thm:lsc}  and that, in addition, 
if $u\in \Theta(\Om,\Gamma)\cap L^\infty(\R^N)$,
then 
$$
u_n:=u 1_{\Om_n}\in \Theta(\Om_n,\Gamma_n)
$$ 
with
\begin{equation}
\label{eq:1sup}
\lim_{n\to +\infty}\int_{\partial^*\Omega_n\cup\Gamma_n}[\gamma^2_l(u_n)+\gamma^2_r(u_n)]\:d\hn=
\int_{\partial^*\Omega\cup\Gamma}[\gamma^2_l(u)+\gamma^2_r(u)]\:d\hn.
\end{equation}
Indeed, let
$(\Om,K)$ be a limit configuration of $(\Om_n,\Gamma_n)$ up to a subsequence. Since $u_n\to u$ strongly in $L^2(\R^N)$ and satisfies the bound \eqref{eq:bound-rob}, we infer that $J_u\tsub K$. By choosing $u$ equal to the eigenfunctions $u_1,\dots,u_k$ of $(\Om,\Gamma)$, and since we are assuming that $\Gamma$ is the union of the their jump sets, we deduce $\Gamma\tsub K$. From
$$
\Prob{\Om}{K}\le \liminf_{n\to +\infty}\Prob{\Om_n}{\Gamma_n}=\Prob{\Om}{\Gamma},
$$
we get $K\teq\Gamma$. The lower semicontinuity \eqref{eq:superficie} entails
\begin{equation}
\label{eq:liminfT}
\int_{\partial^*\Omega\cup\Gamma}[\gamma_l(u)^2+\gamma_r(u)^2]\:d\hn\le\liminf_{n\to+\infty}\int_{\partial^*\Omega_n\cup\Gamma_n}[\gamma_l^2(u_n)+\gamma_r^2(u_n)]\:d\hn,
\end{equation}
while the same argument applied to the function
$$
w:=\sqrt{(\|u\|_\infty+1)^2-u^2}1_\Om
$$
yields
\begin{multline*}
(\|u\|_\infty+1)^2 \Prob{\Om}{\Gamma}-\int_{\partial^*\Om \cup \Gamma}[\gamma^2_l(u)+\gamma^2_r(u)]\,d\hn\\
\le 
\liminf_{n\to+\infty}\left[(\|u\|_\infty+1)^2  \Prob{\Om_n}{\Gamma_n}-\int_{\partial^*\Om_n\cup \Gamma_n}[\gamma^2_l(u_n)+\gamma_r^2(u_n)]\,d\hn\ \right]
\end{multline*}
so that the convergence of the perimeter \eqref{eq:convprob1} entails
$$
\limsup_{n\to +\infty}\int_{\partial^*\Omega_n\cup\Gamma_n}[\gamma^2_l(u_n)+\gamma^2_r(u_n)]\:d\hn\le
\int_{\partial^*\Omega\cup\Gamma}[\gamma^2_l(u)+\gamma^2_r(u)]\:d\hn.
$$
which together with \eqref{eq:liminfT} yields \eqref{eq:1sup}.
\par
Let us come to claim \eqref{eq:lambdaksup} concerning the eigenvalues. Recall that
$$
\lkb(\Om,\Gamma)=\max_{V_k}R_\beta=R_{\beta}(u_k),
$$
where $R_k$ is the Rayleigh quotient \eqref{eq:Rbeta} and $V_k$ is the space generated by the first  $k$ eigenfunctions $u_1,\dots,u_k$. We deduce that $u_{i,n}:=u_i1_{\Om_n}\in \Theta(\Om_n,\Gamma_n)$ for $i=1,\dots,k$, with
\begin{equation}
\label{eq:vin}
u_{i,n}\to u_i\qquad\text{strongly in }L^2(\R^N),
\end{equation}
\begin{equation}
\label{eq:nablavin}
\nabla v_{i,n}\to \nabla u_i\qquad\text{strongly in }L^2(\R^N;\R^N),
\end{equation}
and for every $a_1,\dots,a_k\in\R$
\begin{multline}
\label{eq:convsupn}
\lim_{n\to +\infty} \int_{\partial \Om_n\cup \Gamma_n}\left[\gamma_l^2\left(\sum_i a_i u_{i,n}\right)+\gamma_r^2\left(\sum_i a_i u_{i,n}\right)\right]\,d\hn\\
=\int_{\partial \Om\cup \Gamma}\left[\gamma_l^2\left(\sum_i a_i u_i\right)+\gamma_r^2\left(\sum_i a_i u_i\right)\right]\,d\hn.
\end{multline}
The convergence of the surface energies follows from \eqref{eq:1sup} since the eigenfunctions are in $L^\infty$. If we set
$$
V_k^n:=\text{span}\{u_{1,n},\dots,u_{k,n}\},
$$
we get that $V_k^n\subset \Theta(\Om_n,\Gamma_n)$ is $k$-dimensional for $n$ large and
$$
\limsup_{n\to+\infty}\max_{V_k^n} R_\beta\le \max_{V_k}R_\beta.
$$
Indeed we can assume
$$
\max_{V_k^n} R_\beta=R_\beta\left(\sum_i a_{i,n}u_{i,n}\right)
$$
for some $a_{i,n}\to a_i$ with $|(a_1,\dots,a_k)|=1$. Then \eqref{eq:vin}, \eqref{eq:nablavin} and \eqref{eq:convsupn} together with the uniform bound in $L^\infty$ and the fact $\hn(\partial \Om_n\cup\Gamma_n)\le C$ imply that
$$
R_\beta\left(\sum_i a_{i,n}u_{i,n}\right)\to R_\beta\left(\sum_i a_iu_i\right)\le \max_{V_k}R_\beta.
$$
We can thus write
$$
\limsup_{n\to+\infty}\lambda_k(\Om_n,\Gamma_n)\le \limsup_{n\to+\infty}\max_{V_k^n} R_\beta\le \max_{V_k}R_\beta=\lambda_k(\Om,\Gamma),
$$
and claim \eqref{eq:lambdaksup} follows.

\vskip10pt\noindent{\bf Step 2.} In view of Step 1, it is not restrictive to assume $\Om\subset\R^N$ open with smooth boundary. 
We apply Theorem \ref{teo:cor_toa} to approximate the vector valued function on $\Om$
$$
u:=(u_1,\dots,u_k)\in SBV(\Om;\R^k)\cap L^\infty(\Om;\R^N),
$$
where $u_1,\dots,u_k$ are the first $k$ eigenfunctions of $(\Om,\Gamma)$.
There exists $v^n=(v^n_1,\dots, v^n_k)\subset L^2(\Om;\R^N)$ with $\|v^n\|_\infty \le \|u\|_\infty$ such that
$$
v^n\to u\quad\text{strongly in $L^2(\Om;\R^k)$},
$$
\begin{equation}
\label{eq:CTnabla}
\nabla v^n\to \nabla u\quad\text{strongly in $L^2(\Om;\R^{kN})$},
\end{equation}
$$
J_{v^n}\ \text{polyhedral in $\Om$}
$$
$$
v^n\in W^{m,\infty}(\Om\setminus\overline{J_{v^n}};\R^N)\qquad\text{for every $m\ge 1$},
$$
and, for any upper semicontinuous function $\varphi:\Omb\times \R^N\times \R^N\times S^{N-1}\to [0,+\infty[$ locally bounded near the boundary and any open set $A\subseteq \Om$
\begin{align}
\label{eq:CTvarphi}
\limsup_{n\to+\infty}\int_{J_{v^n}\cap A}\varphi(x,\gamma_l(v^n),\gamma_r(v^n),\nu_{J_{v^n}})\:d\hn
\le\int_{J_{v}\cap A}\varphi(x,\gamma_l(u),\gamma_r(u),\nu_{J_{u}})\:d\hn.
\end{align}
Since the edges of $(N-1)$-dimensional simplexes have two capacity zero, we can assume that $J_{v^n}$ is composed of a finite family of disjoint simplexes compactly contained in $\Om$. 
\par
Notice that (choose $\varphi=1$)
\begin{equation}
\label{eq:CTj}
\limsup_{n\to+\infty}\hn(\overline{J_{v^n}})\le\hn(J_{u})=\hn(\Gamma).
\end{equation}
Moreover
$$
|Dv^n|(\Om)\to |Du|(\Om) 
$$ 
This follows from \eqref{eq:CTnabla} and \eqref{eq:CTvarphi} with the choice $\varphi(x,a,b,\nu)=|a-b|$ and $A=\Om$. As a consequence we obtain the convergence of the associated traces on $\partial\Om$, since the trace operator is continuous under strong $L^1$ convergence together with the convergence of the total variation (it is the so called strict convergence in $BV$): the convergence of the traces holds also in $L^2$ in view of the uniform bound on $\|v^n\|_\infty$.
\par
By choosing $\varphi(x,a,b,\nu)=|\alpha\cdot a|^2+|\alpha\cdot b|^2$ where $\alpha\in \R^k$, and using the convergence of the traces on $\partial\Om$ we may write
\begin{multline*}
\limsup_n \int_{\partial\Om \cup J_{v^n}} [\gamma_l^2(\alpha_1 v^n_1+\dots+\alpha_k v^n_k)+\gamma_r^2(\alpha_1 v^n_1+\dots+\alpha_k v^n_k)]\,d\hn\\
\le \int_{\partial\Om \cup \Gamma} [\gamma_l^2(\alpha_1 u_1+\dots+\alpha_k u_k)+\gamma_r^2(\alpha_1u_1+\dots+\alpha_k u_k)]\,d\hn,
\end{multline*}
the inequality being uniform in $\alpha$ for $|\alpha|=1$: this is due to the uniform bound in $L^\infty$ for $v^n$ together with \eqref{eq:CTj}.
\par
In view of the geometric structure of $\overline{J_{v^n}}$ we can write
$$
\overline{J_{v^n}}=\bigcap_k A^k_n
$$
where $(A_n^k)_{k\in\N}$ is a decreasing sequence of smooth open set compactly contained in $\Om$ with 
$$
\lim_{k\to +\infty}\hn(\partial A^k_n)=2\hn(\overline{J_{v^n}}).
$$
In view of the regularity of $v^n$ outside $\overline{J_{v^n}}$, by using a diagonal argument we find $k=k_n$ such that
\begin{multline}
\label{eq:CTAn}
\limsup_n \int_{\partial\Om \cup \partial A_n^{k_n}} \gamma_l^2(\alpha_1 v^n_1+\dots+\alpha_k v^n_k)\,d\hn\\
\le \int_{\partial\Om \cup \Gamma} [\gamma_l^2(\alpha_1 u_1+\dots+\alpha_k u_k)+\gamma_r^2(\alpha_1u_1+\dots+\alpha_k u_k)]\,d\hn,
\end{multline}
the inequality being uniform in $\alpha\in \R^N$ with $|\alpha|=1$, and
$$
\limsup_{n\to +\infty} \hn(\partial A_n^{k_n})\le 2\limsup_{n\to+\infty}\hn(\overline{J_{v^n}})=2\hs^{d.1}(\Gamma).
$$
\par
By considering the smooth set $\Om_n:=\Om\setminus A_n^{k_n}$ we have
$$
\limsup_{n\to +\infty}\hn(\partial \Om_n)=\hn(\partial\Om)+\limsup_n \hn(\partial A_n^{k_n})\le \hn(\partial\Om)+2\hs^{d.1}(\Gamma)=\Prob{\Om}{\Gamma}.
$$
Then in view of \eqref{eq:CTAn}, by using the same arguments of Step 1 we infer that
$$
\limsup_{n\to+\infty}\lambda_{k,\beta}(\Om_n)\le \lkb(\Om,\Gamma)
$$
and the proof is concluded.
\end{proof}

\begin{remark}
\label{rem:approx-lek}
{\rm
Notice that the previous construction shows that indeed
$$
\limsup_{n\to+\infty}\lambda_{h,\beta}(\Om_n)\le \tilde \lambda_{h,\beta}(\Om,\Gamma)
$$
for every $1\le h\le k$, as the function $u$ involved in Step 2 is such that its components  $(u_1,\dots,u_k)$ are the first $k$ eigenfunctions of $(\Om,\Gamma)$.
}
\end{remark}

\section{Proof of the main results}
\label{sec:proofs}
In this section we provide the proof of the main results of the paper.

\subsection{Proof of Theorem \ref{th:main1}}
\label{sec:main1}
Let us start with the following general inequality, which is based on a cutting argument used in \cite[Theorem 4.6]{BucGiac-lambdak}, 

\begin{lemma}
\label{lem:cut}
Let $(\Omega,\Gamma) \in \as(\R^N)$, and for $t\in \R$
$$
\Omega_t:=\Omega\cap\left\{x_1<t\right\},\qquad\text{and}\qquad \Gamma_t:=\Omega\cap\left\{x_1<t\right\}.
$$ 
Then for a.e. $t$ large enough we have
\begin{equation}\label{eq:cut}
\lkb(\Omega_t,\Gamma_t)\le \lkb(\Omega,\Gamma)+C\hn(\Omega\cap\{x_1=t\}),
\end{equation}
where $C>0$ is independent of $t$.
\end{lemma}

\begin{proof}
We can assume $|\Om\setminus \Om_t|>0$ for every $t\in\R$, i.e., $\Om$ is unbounded in the positive $x_1$ direction.
\par
Let $u_1,\ldots,u_k\in \Theta(\Om,\Gamma)$ be the first $k$ eigenfunctions of $(\Om,\Gamma)$, which we can assume to form a $L^2$-orthonormal basis. 
For every $t\in\R$, we define the functions $u_{j,t}\in\Theta(\Omega_t,\Gamma_t)$ by setting, for each $j=1,\ldots,k$
$$
u_{j,t}:=u_j1_{\left\{x_1<t\right\}}.
$$
The functions $u_{j,t}$ are linearly independent for $t$ sufficiently large. Let us consider the coefficients $\alpha_{1,t},\ldots,\alpha_{k,t}\in\R$ with $\sum_{j=1}^k\alpha^2_{j,t}=1$. such that the function
$$
U_t:=\sum_{j=1}^k\alpha_{j,t} u_{t,j}
$$
realizes the maximum of the Rayleigh quotient $R_\beta$ on $\text{span}\{u_{1,t},\ldots,u_{k,t}\}$. We get
\begin{multline*}
\lkb(\Omega_t,\Gamma_t)\le\frac{\displaystyle\int_{\Omega_t}|\nabla U_t|^2\:dx+\beta\int_{\partial^*\Omega_t\cup\Gamma_t}U_t^2\:d\hn}{\displaystyle\int_{\Omega_t}U_t^2\:dx}\\ =\frac{\displaystyle\int_{\Omega_t}|\nabla U_t|^2\:dx+\beta\int_{\partial^*\Omega_t\cup\Gamma_t}U_t^2\:d\hn}{\displaystyle 1-\int_{\Omega\setminus\Omega_t}U_t^2\:dx}
\end{multline*}
in view of the $L^2$-orthonormality of $u_1,\ldots,u_k$.
Moreover since
$$
\int_{\Omega\setminus\Omega_t}U_t^2\:dx\le 2\sum_{j=1}^k\int_{\Omega\setminus\Omega_t}u_j^2\:dx\to 0
$$
as $t\to+\infty$, we can assume that for $t$ sufficiently large
$$
\frac{1}{\displaystyle 1-\int_{\Omega\setminus\Omega_t}U_t^2\:dx}\le1+2\int_{\Omega\setminus\Omega_t}U_t^2\:dx\le 2.
$$
We thus obtain the following estimates
\begin{multline}\label{eq:4.27,5bis}
\lkb(\Omega_t,\Gamma_t)\le\left(1+2\int_{\Omega\setminus\Omega_t}U_t^2\:dx\right)\left(\int_{\Omega_t}|\nabla U_t|^2\:dx+\beta\int_{\partial^*\Omega_t\cup\Gamma_t}U_t^2\:d\hn\right)\\
\le\left(1+2\int_{\Omega\setminus\Omega_t}U_t^2\:dx\right) \cdot 
\Bigg(\lkb(\Omega,\Gamma)-\int_{\Omega\setminus\Omega_t}|\nabla U_t|^2\:dx\\
-\beta\int_{(\partial^*\Omega\cup\Gamma)\cap\left\{x_1>t\right\}}U_t^2\:d\hn 
+\beta\int_{\left\{x_1=t\right\}}U_t^2\:d\hn\Bigg)\\
\le \lkb(\Omega,\Gamma)+2\lkb(\Omega,\Gamma)\int_{\Omega\setminus\Omega_t}U_t^2\:dx
-\int_{\Omega\setminus\Omega_t}|\nabla U_t|^2\:dx
-\beta\int_{(\partial^*\Omega\cup\Gamma)\cap\left\{x_1>t\right\}}U_t^2\:d\hn\\
+2\beta\int_{\left\{x_1=t\right\}}U_t^2\:d\hn.
\end{multline}
Let us consider the restrictions of the functions $u_1,\ldots,u_k$ to $\Omega\setminus\Omega_t$ and let us reflect them across the hyperplane $\left\{x_1=t\right\}$. We obtain the functions $w_{j,t}\in \Theta(A_t, K_t)$, where $A_t$ is obtained by symmetrizing $\Om\setminus \{x_1\le t\}$, while $K_t$ is obtained 
by symmetrizing $\Gamma\setminus \{x_1\le t\}$.
Setting $W_t:=\sum_{j=1}^k\alpha_{j,t}w_{j,t}$, in view of the Faber-Krahn inequality \eqref{eq:faber} we may write
\begin{multline*}
\lambda_{1,\beta}(B(t))\le\frac{\displaystyle\int_{ A_t}|\nabla W_t|^2\:dx+\beta\int_{ \partial^* A_t\cup K_t}W_t^2\:d\hn}{\displaystyle\int_{ A_t}W_t^2\:dx}\\
\le\frac{\displaystyle2\int_{\Omega\setminus\Omega_t}|\nabla U_t|^2\:dx+2\beta\int_{(\partial^*\Omega\cup\Gamma)\cap\left\{x_1>t\right\}}U_t^2\:d\hn}{\displaystyle2\int_{\Omega\setminus\Omega_t} U_t^2\:dx}\\
=\frac{\displaystyle\int_{\Omega\setminus\Omega_t}|\nabla U_t|^2\:dx+\beta\int_{(\partial^*\Omega\cup\Gamma)\cap\left\{x_1>t\right\}}U_t^2\:d\hn}{\displaystyle\int_{\Omega\setminus\Omega_t}U_t^2\:dx},
\end{multline*}
where $B(t)$ is the ball of $\R^N$ such that $|B(t)| =|A_t|=2|\Omega\setminus\Omega_t|$, so that
$$
\int_{\Omega\setminus\Omega_t}|\nabla U_t|^2\:dx+\beta\int_{(\partial^*\Omega\cup\Gamma)\cap\left\{x_1>t\right\}}U_t^2\:d\hn\ge \lambda_{1,\beta}(B(t))\int_{\Omega\setminus\Omega_t}U_t^2\:dx.$$
Then, looking at the estimate
\begin{align*}
2\lkb(\Omega,\Gamma)&\int_{\Omega\setminus\Omega_t}U_t^2\:dx-\int_{\Omega\setminus\Omega_t}|\nabla U_t|^2\:dx-\beta\int_{(\partial^*\Omega\cup\Gamma)\cap\left\{x_1>t\right\}}U_t^2\:d\hn\\
&\le[2\lkb(\Omega,\Gamma)- \lambda_{1,\beta}(B(t))]\int_{\Omega\setminus\Omega_t}U_t^2\:dx,
\end{align*}
we deduce that the right hand side is strictly negative for $t$ large enough (since the quantity $\lambda_{1,\beta}(B(t))$ diverges as the measure of $B(t)$ goes to zero). In conclusion, coming back to \eqref{eq:4.27,5bis}, we obtain 
$$
\lkb(\Omega_t,\Gamma_t)\le \lkb(\Omega,\Gamma)+2\beta\int_{\left\{x_1=t\right\}}U^2_t\:d\hn,
$$
and claim \eqref{eq:cut} follows since $U_t$ is bounded in $L^\infty$ by a constant independent of t (as the eigenfunctions $u_j$ are bounded in $L^\infty$).
\end{proof}

We can now state the following boundedness property for configurations which are minimizers of the problem.

\begin{proposition}[\bf Boundedness of minimizers]
\label{pro:bounded2}
Let $(\Omega,\Gamma)\in \as(\R^N)$ be a minimizer for Problem \eqref{eq:jumprobinconst}. Then $\Omega$ is bounded.
\end{proposition}

\begin{proof}
 Let $(\Omega,\Gamma)$ be a minimizer for \eqref{eq:jumprobinconst} and let us suppose that $\Omega$ is unbounded. It is not restrictive to assume, up to translations and rotations, that $\Om$ is unbounded in the positive direction $x_1$. 
 \par
Let us set
$$
\Omega_t:=\Omega\cap\left\{x_1<t\right\},\qquad\text{and}\qquad \Gamma_t:=\Omega\cap\left\{x_1<t\right\}.
$$ 
Clearly $(\Om_t,\Gamma_t)\in \as(\R^N)$, and by projection on the hyperplane $x_1=t$ we get
$$
\Prob{\Omega_t}{\Gamma_t}\le \Prob{\Omega}{\Gamma}.
$$
We divide the proof in two steps.
 
\vskip10pt\noindent{\bf Step 1.}
We claim that for a.e. $t$ large enough
\begin{equation}
\label{eq:claim-om}
\hn(\partial^*(\Omega \setminus \Om_t))\le C_1\hn(\Omega\cap\{x_1=t\}),
\end{equation}
where $C_1$ is independent of $t$.
\par
Indeed, letting
$$
\eta(t):=\left(\frac{\Prob{\Omega}{\Gamma}}{\Prob{\Omega_t}{\Gamma_t}}\right)^{\frac{1}{N-1}}\ge 1
$$
and considering the dilated couple $(\tilde{\Omega}_t,\tilde{\Gamma}_t)\in \as(\R^N)$ with
$$
\tilde{\Omega}_t:=\eta(t)\Omega_t,\qquad\text{and}\qquad \tilde{\Gamma}_t:=\eta(t)\Gamma_t,
$$
the optimality of $(\Omega,\Gamma)$ and the admissibility of $(\tilde{\Omega}_t,\tilde{\Gamma}_t)$ for Problem \eqref{eq:jumprobinconst} yield 
\begin{equation}
\label{eq:lkb-dil}
\lkb(\Omega,\Gamma)\le\lkb(\tilde{\Omega}_t,\tilde{\Gamma}_t)=\frac{1}{\eta(t)^2}\tilde \lambda_{k,\eta(t)\beta}(\Omega_t,\Gamma_t)\le\frac{1}{\eta(t)}
\lkb(\Omega_t,\Gamma_t),
\end{equation}
where we used the rescaling property of Remark \ref{rem:scaling}.
Since for a.e. $t\in\R$
$$
\Prob{\Omega_t}{\Gamma_t}=\Prob{\Omega}{\Gamma}-\hn((\partial^*\Omega\cup\Gamma)\cap\{x_1 > t\})+\hn(\Omega\cap\{x_1=t\}),
$$
by the very definition of $\eta(t)$ we get
$$\eta(t)=\left(1+\frac{\hn((\partial^*\Omega\cup\Gamma)\cap\{x_1>t\})-\hn(\Omega\cap\{x_1=t\})}{\Prob{\Omega_t}{\Gamma_t}}\right)^{\frac{1}{N-1}}.$$
Since by projection on the hyperplane $x_1=t$
$$
\hn((\partial^*\Omega\cup\Gamma)\cap\{x_1>t\})\ge\hn(\partial^*\Omega\cap\{x_1>t\})\ge\hn(\Omega\cap\{x_1=t\}),
$$
we deduce
$$
\hn((\partial^*\Omega\cup\Gamma)\cap\{x_1>t\})-\hn(\Omega\cap\{x_1=t\})\to 0
$$
as $t\to+\infty$. Hence we get for a.e. $t$ large enough
\begin{equation}
\label{eq:eta}
\eta(t)\ge 1+ C_2 \left(\hn((\partial^*\Omega\cup\Gamma)\cap\{x_1>t\})-\hn(\Omega\cap\{x_1=t\})\right)
\end{equation}
for some positive constant $C_2$ independent of $t$.
In view of estimates \eqref{eq:lkb-dil} and \eqref{eq:cut} we thus get
\begin{equation*}
\eta(t)\lkb(\Omega,\Gamma)\le\lkb(\Omega_t,\Gamma_t)\le\lkb(\Omega,\Gamma)+C\hn(\Omega\cap\{x_1=t\})
\end{equation*}
so that using \eqref{eq:eta} we infer
$$
\hn((\partial^*\Omega\cup\Gamma)\cap\{x_1>t\})-\hn(\Omega\cap\{x_1=t\})\le C_3\hn(\Omega\cap\{x_1=t\})
$$
for a suitable $C_3>0$  independent of $t$, which readily implies claim \eqref{eq:claim-om}.

\vskip10pt\noindent{\bf Step 2.}
In view of \eqref{eq:claim-om}, using the isoperimetric inequality we get for a.e. $t$ large enough
$$
|\Omega\setminus\Omega_t|^{\frac{N-1}{N}}\le C_4\hn(\Omega\cap\{x_1=t\}),
$$
for some $C_4>0$ independent of $t$.
By setting $g(t):=|\Omega\setminus\Omega_t|$ we get for a.e. $t$ that $\hn(\Omega\cap\{x_1=t\})=-g'(t)$ and thus (recall that we are assuming by contradiction $g(t)\not=0$ for $t$ sufficiently large)
$$
\frac{g'(t)}{g(t)^{\frac{N-1}{N}}}\le-\frac{1}{C_4}.
$$
But then if $t_0$ is sufficiently large
$$
-Ng(t_0)^{\frac{1}{N}}=\int_{t_0}^{+\infty}\frac{g'(t)}{g(t)^{\frac{N-1}{N}}}\:dt=-\infty
$$
which is a contradiction.
\end{proof}

%
We are now in a position to prove Theorem \ref{th:main1}.

\begin{proof}[Proof of Theorem \ref{th:main1}]
Using the monotonicity under dilations of Remark \ref{rem:scaling}, problem \eqref{eq:jumprobinconst} is equivalent to
\begin{equation}\label{eq:jumprobinconst2}
\min\left\{\lkb(\Omega,\Gamma): (\Omega,\Gamma)\in\mathcal{A}(\R^N) \text{ with }\Prob{\Omega}{\Gamma} \le p\right\}.
\end{equation}
If minimizers $(\Om,\Gamma)$ exist, then clearly $\Prob{\Om}{\Gamma}=p$, while $\Om$ is bounded according to Proposition \ref{pro:bounded2}. Finally, property \eqref{eq:den1} follows from Theorem \ref{thm:den}.
\par
Let us proceed, as usual for these optimization problems, by induction on $k\in\N$. For $k=1$, according to Remark \ref{rem:1ball}, minimizers are balls of perimeter $p$.
\par
Let us now assume that a minimizer exists for every $j<k$. Let $(\Omega_n,\Gamma_n)_n$ be a minimizing sequence for problem \eqref{eq:jumprobinconst2}. 
We can assume up to a subsequence 
$$
\lim_{n\to+\infty}|\Omega_n|=m,\quad \text{with $0<m<+\infty$}.
$$
Indeed, thanks to the isoperimetric inequality we have the upper bound
$$
|\Omega_n|^\frac{N-1}{N}\le C Per(\Omega_n)\le C \Prob{\Omega_n}{\Gamma_n}\le Cp.
$$
On the other hand, if $|\Omega_n|$ vanishes, we would obtain thanks to the Faber-Krahn inequality \eqref{eq:faber}
$$
\lkb(\Omega_n,\Gamma_n)\ge\tilde \lambda_{1,\beta}(\Omega_n,\Gamma_n)\ge\lambda_{1,\beta}(B_n)\to +\infty,
$$
where $B_n$ is the ball having the same measure of $\Omega_n$,  against the fact that $(\Om_n,\Gamma_n)$ is a minimizing sequence.
\par
Let us apply a concentration-compactness argument to the sequence $(1_{\Omega_n})_{n\in\N}$. For every $r>0$ let us consider the monotone increasing functions $\alpha_n:[0,+\infty[\to [0,+\infty[$
\begin{equation*}
\alpha_n(r):=\sup_{y \in \R^N}|\Om_n \cap Q_r(y)|,
\end{equation*}
where $Q_r(y)$ is the cube centered at $y$ with side $r$. Up to a subsequence, in view of Helly's theorem, we may assume that
\begin{equation*}
\alpha_n \to \alpha
\qquad\text{pointwise on }[0,+\infty[
\end{equation*}
for a suitable monotone increasing function $\alpha:[0,+\infty[\to [0,+\infty[$. 
\par
The following situations may occur.
\begin{itemize}
\item[(a)] {\it Vanishing}: $\lim_{r\to+\infty}\alpha(r)=0$;
\item[(b)] {\it Dichotomy}: $\lim_{r\to+\infty}\alpha(r)=\bar \alpha \in ]0,m[$;
\item[(c)] {\it Compactness}: $\lim_{r\to+\infty}\alpha(r)=m$.
\end{itemize}
Let us deal with the three cases separately.

\vskip10pt\noindent{\bf Step 1: Vanishing cannot occur.} Indeed, if it was the case, one would have for every $r>0$
\begin{equation}
\label{eq:starvanish}
\sup_{y\in\R^N}\left|\Omega_n\cap Q_r(y)\right|\to 0.
\end{equation}
Let $u_{k,n}$ be a $L^2$-normalized $k$-th eigenfunction of $(\Omega_n,\Gamma_n)$. Since
\begin{multline*}
\int_{\R^N}|\nabla u_{k,n}|^2\:dx+\int_{J_{u_{n,k}}}\left[\gamma^2_l(u_{k,n})+\gamma^2_r(u_{k,n})\right]\:d\hn \\
\le \int_{\Omega_n}|\nabla u_{k,n}|^2\:dx+\int_{\partial\Omega_n\cup\Gamma_n}\left[\gamma^2_l(u_{k,n})+\gamma^2_r(u_{k,n})\right]\:d\hn=\lambda_{k,\beta}(\Omega_n,\Gamma_n)\le C,
\end{multline*}
in view of \cite[Lemma 4]{BucGiac} applied to both the negative and positive parts of $u_{k,n}$, there exists $y_n\in\R^N$ such that
$$
\left|\Omega_n\cap Q_1(y_n)\right| \ge \left|supp(u_{k,n})\cap Q_1(y_n)\right|\ge C'_N\left(\frac{1}{2C+2}\right)^N>0,
$$
against \eqref{eq:starvanish}.

\vskip10pt\noindent{\bf Step 2: Compactness.} If compactness occurs, there exists a set of finite perimeter $\Omega \subset\R^N$ such that
$$
1_{\Omega_n}\to 1_\Omega \qquad\text{strongly in }L^1(\R^N).
$$ 
Let $\Gamma$ be given by Theorem \ref{thm:lsc}, so that $(\Om,\Gamma)\in \as(\R^N)$ with
$$
\Prob{\Om}{\Gamma}\le \liminf_n \Prob{\Om_n}{\Gamma_n}\le p
$$
and, according to Theorem \ref{pro:lscrkb},
$$
\lambda_{k,\beta}(\Omega,\Gamma)\le\liminf_{n\to+\infty}\lambda_{k,\beta}(\Omega_n,\Gamma_n).
$$
We infer that $(\Omega,\Gamma)$ is a minimizer for problem \eqref{eq:jumprobinconst2}.

\vskip10pt\noindent{\bf Step 3: Dichotomy.} 
Let dichotomy occur. Then there exists $\tilde\alpha\in]0,m[$ such that the following assertion holds true: we can find $x_n\in\R^N$ and $0<r_n<R_n$, $R_n-r_n\to+\infty$, such that setting
$$
\Omega_{n,1}:=\Omega_n\cap B_{r_n}(x_n),\qquad \Gamma_{n,1}:=\Gamma_n\cap B_{r_n}(x_n)
$$
and
$$
\Omega_{n,2}:=\Omega_n\setminus B_{R_n}(x_n),\qquad \Gamma_{n,2}:=\Gamma_n\setminus B_{R_n}(x_n),
$$
we have
$$
\left||\Omega_{n,1}|-\tilde\alpha\right|\to 0,\qquad \left||\Omega_{n,2}|-(m-\tilde\alpha)\right|\to 0,
$$
with
$$
\hn(\Omega_n\cap\partial B_{r_n}(x_n))\to 0,\qquad \hn(\Omega_n\cap\partial B_{R_n}(x_n))\to 0.
$$
Notice that 
\begin{equation}\label{eq:dc2}
\Prob{\Omega_n}{\Gamma_n}\ge \Prob{\Omega_{n,1}}{\Gamma_{n,1}}+\Prob{\Omega_{n,2}}{\Gamma_{n,2}}-\e_n
\end{equation}
with $\e_n\to 0$. Up to a subsequence we may assume
\begin{equation}
\label{eq:dcper}
\Prob{\Omega_{n,1}}{\Gamma_{n,1}} \to p_1>0\qquad \text{and}\qquad
\Prob{\Omega_{n,2}}{\Gamma_{n,2}} \to p_2>0
\end{equation}
with $p_1+p_2\le p$.
\par
Now, by testing the Rayleigh quotient for $\Omega_{n,1}\cup\Omega_{n,2}$ on the eigenfunctions of $\Omega_n$, taking into account their uniform boundedness in $L^\infty$ given by Theorem \ref{th:boundeigen} and Remark \ref{rem:disjoint} which characterizes the spectrum of $\Om_n^1\cup \Om_n^2$ we get
\begin{equation}\label{eq:dc1}
\begin{split}
\lkb(\Omega_n,\Gamma_n)&\ge\lkb(\Omega_{n,1}\cup\Omega_{n,2},\Gamma_{n,1}\cup\Gamma_{n,2})-\delta_n\\
&=\min_{i=0,\ldots,k}\max\left\{\tilde\lambda_{i,\beta}(\Omega_{n,1},\Gamma_{n,1}),\tilde\lambda_{k-i,\beta}(\Omega_{n,2},\Gamma_{n,2})\right\}-\delta_n\\
&=\max\left\{\tilde \lambda_{\bar{i},\beta}(\Omega_{n,1},\Gamma_{n,1}),\tilde \lambda_{k-{\bar{i}},\beta}(\Omega_{n,2},\Gamma_{n,2})\right\}-\delta_n,
\end{split}
\end{equation}
where $\delta_n\to 0$ and $\bar{i}$ is independent of $n$ (up to subsequences).
\par
We point out that $\bar{i}<k$. Otherwise we would have
$$
\lkb(\Omega_n,\Gamma_n) \ge \lkb(\Omega_{n,1},\Gamma_{n,1}) -\delta_n
$$
and, in view of \eqref{eq:dc2} and \eqref{eq:dcper} ,
$$
\Prob{\Omega_{n,1}}{\Gamma_{n,1}})+\e<\Prob{\Omega_n}{\Gamma_n},
$$
for some $\e>0$: this contradicts the fact that $(\Omega_n,\Gamma_n)$ is a minimizing sequence in view of the monotonicity under dilations.
\par
Let $(\Omega_1,\Gamma_1)$ and $(\Omega_2,\Gamma_2)$ be minimizing couples for problem \eqref{eq:jumprobinconst2} relative to $\tilde \lambda_{\bar{i},\beta}$ with perimeter constraint $p_1$ and $\tilde \lambda_{k-{\bar{i}},\beta}$ with perimeter constraint $p_2$ respectively, whose existence is guaranteed by our induction assumption. We claim that
\begin{equation}\label{eq:dc3}
\tilde \lambda_{\bar{i},\beta}(\Omega_1,\Gamma_1)\le\liminf_{n\to+\infty}\tilde \lambda_{\bar{i},\beta}(\Omega_{n,1},\Gamma_{n,1})
\end{equation}
and
\begin{equation}\label{eq:dc4}
\tilde\lambda_{k-\bar{i},\beta}(\Omega_2,\Gamma_2)\le\liminf_{n\to+\infty}\tilde\lambda_{k-\bar{i},\beta}(\Omega_{n,2},\Gamma_{n,2}).
\end{equation}
Since $\Om_1,\Om_2$ are bounded sets by Proposition \ref{pro:bounded2}, we can assume that they are at positive distance. 
Let us set
$$
\Omega:=\Omega_1\cup\Omega_2\qquad\text{and}\qquad \Gamma:=\Gamma_1\cup\Gamma_2.
$$
Clearly $(\Om,\Gamma)\in \as(\R^N)$ with
$$
\Prob{\Omega}{\Gamma}=\Prob{\Omega_1}{\Gamma_1}+\Prob{\Omega_2}{\Gamma_2}=p_1+p_2\le p,
$$
while thanks to Remark \ref{rem:disjoint}
$$
\lkb(\Omega,\Gamma)\le\max\left\{\tilde\lambda_{\bar{i},\beta}(\Omega_1,\Gamma_1),\tilde\lambda_{k-{\bar{i}},\beta}(\Omega_2,\Gamma_2)\right\}.
$$
We infer that $(\Omega,\Gamma)$ is an admissible couple for the minimization of $\lkb$ and, taking into account \eqref{eq:dc1} and claims \eqref{eq:dc3}, \eqref{eq:dc4} we have
\begin{multline*}
\lkb(\Omega,\Gamma)\le\max\left\{\tilde\lambda_{\bar{i},\beta}(\Omega_1,\Gamma_1),\tilde\lambda_{k-{\bar{i}},\beta}(\Omega_2,\Gamma_2)\right\}\\ 
\le\liminf_{n\to+\infty}\left(\max\left\{\tilde\lambda_{\bar{i},\beta}(\Omega_{n,1},\Gamma_{n,1}),\tilde\lambda_{k-{\bar{i}},\beta}(\Omega_{n,2},\Gamma_{n,2})\right\}\right)\\
\le\liminf_{n\to+\infty}\lkb(\Omega_n,\Gamma_n).
\end{multline*}
We conclude that $(\Omega,\Gamma)$ is a minimizer for Problem \eqref{eq:jumprobinconst2}. 
\par
In order to conclude, we need to check claims \eqref{eq:dc3} and \eqref{eq:dc4}. Let us prove the first one, the other being similar.
\par
If $\Prob{\Omega_{n,1}}{\Gamma_{n,1}}\le p_1$, there is nothing to prove in view of the minimality of $\Omega_1$. On the other hand, if $\Prob{\Omega_{n,1}}{\Gamma_{n,1}}>p_1$, by \eqref{eq:dcper} there exists $\varepsilon_n>0$, $\varepsilon_n\to 0$, such that $\Prob{\Omega_{n,1}}{\Gamma_{n,1}}=p_1+\varepsilon_n$. Let us define the quantity
$$t_n:=\left(\frac{p_1}{p_1+\varepsilon_n}\right)^{\frac{1}{N-1}}<1$$
and the sets
$$
\tilde{\Omega}_{n,1}:=t_n\Omega_{n,1},\qquad \tilde{\Gamma}_{n,1}:=t_n\Gamma_{n,1}.
$$
Clearly $\Prob{\tilde{\Omega}_{n,1}}{\tilde{\Gamma}_{n,1}})=p_1$, so that
\begin{equation}\label{eq:dc5}
\tilde\lambda_{\bar{i},\beta}(\Omega_1,\Gamma_1)\le\tilde\lambda_{\bar{i},\beta}(\tilde{\Omega}_{n,1},\tilde{\Gamma}_{n,1})
\end{equation}
thanks the minimality of $(\Omega_1,\Gamma_1)$. On the other hand, in view of the scaling property given by Remark \ref{rem:scaling}
it holds
\begin{equation}\label{eq:dc6}
\tilde\lambda_{\bar{i},\beta}(\tilde{\Omega}_{n,1},\tilde{\Gamma}_{n,1})=\frac{1}{t_n^2}\tilde\lambda_{\bar{i},t_n\beta}(\Omega_{n,1},\Gamma_{n,1})\le\frac{1}{t_n^2}\tilde\lambda_{\bar{i},\beta}(\Omega_{n,1},\Gamma_{n,1}).
\end{equation}
Putting together \eqref{eq:dc5} and \eqref{eq:dc6} and observing that $t_n\to 1$ we finally get \eqref{eq:dc3}.
\end{proof}

\subsection{Proof of Theorem \ref{thm:main3}}
\label{sec:main3}

In this section we provide the proof of Theorem \ref{thm:main3}. To this aim, we adapt to our context an induction argument applied, for instance, in \cite{de2014existence,Na}. Using the monotonicity under dilations of Remark \ref{rem:scaling}, problem \eqref{eq:pb-gen} is equivalent to
\begin{equation}\label{eq:pb-gen2}
\min\left\{F(\Omega,\Gamma): (\Omega,\Gamma)\in\mathcal{A}(\R^N) \text{ with }\Prob{\Omega}{\Gamma} \le p\right\}.
\end{equation}
It is convenient to frame \eqref{eq:pb-gen2} within a larger class of problems. 
\par
Let us fix $0\le k_{\ell,2}<k_\ell$ and $0<\gamma_1\le \gamma_2\le \dots \le \gamma_{k_{\ell,2}}$, and let $k_{\ell,1}:=k_\ell-k_{\ell,2}$. For every $(\Om,\Gamma)\in \as(\R^N)$, let us consider the
ordered $k_\ell$-tuple
$$
\mu(\Om,\Gamma):=(\tilde \lambda_{1,\beta}(\Om,\Gamma),\dots, \tilde \lambda_{k_{\ell,1},\beta}(\Om,\Gamma),\gamma_1,\dots, \gamma_{k_{\ell,2}})^{ord}
$$
(in the case $k_{\ell,2}=0$, we mean that no $\gamma$ is involved, so that $\mu(\Om,\Gamma)=(\tilde \lambda_{1,\beta}(\Om,\Gamma),\dots, \tilde \lambda_{k_\ell,\beta}(\Om,\Gamma)$). If we set
$$
F_\gamma(\Om,\Gamma):=f(\mu_{k_1}(\Om,\Gamma),\dots, \mu_{k_\ell}(\Om,\Gamma)),
$$
then problem \eqref{eq:pb-gen2} is a particular case of the following one:
\begin{equation}\label{eq:pb-gen-gamma}
\min\left\{F_\gamma(\Omega,\Gamma): (\Omega,\Gamma)\in\mathcal{A}(\R^N) \text{ with }\Prob{\Omega}{\Gamma} \le p\right\}.
\end{equation}

Considering the more general problem \eqref{eq:pb-gen-gamma} makes it easier to formulate the induction scheme needed to prove existence: more precisely, the dichotomy case for problem \eqref{eq:pb-gen} forces to consider problems of the form \eqref{eq:pb-gen-gamma}.
\par
The following result holds true.

\begin{theorem}
\label{thm:main3-gen}
Problem \eqref{eq:pb-gen-gamma} admits a bounded minimizer.
\end{theorem}

Notice that in principle not every minimizer is bounded: this is because some coefficients are "frozen" and set equal to $\gamma_1,\dots,\gamma_{k_{\ell,2}}$ (think of $F(\Om,\Gamma)=\tilde \lambda_{k,\beta}(\Om,\Gamma)$ and fix $\gamma_1<\min\{\tilde \lambda_{k,\beta}(\Om,\Gamma)\,:\, \Prob{\Om}{\Gamma}\le p\}$).

\begin{proof}
Let us divide the proof in two steps.

\vskip10pt\noindent{\bf Step 1.}
Let us prove that if problem \eqref{eq:pb-gen-gamma} admits a minimizer $(\Om,\Gamma)\in \as(\R^N)$, then it admits also a bounded minimizer. 
\par
For $t\in \R$ let us set
$$
\Om_t:=\Om \cap \{x_1<t\}\qquad\text{and}\qquad \Gamma_t:=\Gamma\cap \{x_1<t\}.
$$
We will show that either $(\Om,\Gamma)$ is bounded in the positive direction $x_1$, or $(\Om_t,\Gamma_t)$ is a minimizer for some $t$ large enough: iterating the argument with the negative $x_1$-direction, and repeating the considerations for the other directions, the result follows.
\par
Let us assume that $\Om$ is unbounded in the positive $x_1$ direction.
Letting
$$
\eta(t):=\left(\frac{\Prob{\Omega}{\Gamma}}{\Prob{\Omega_t}{\Gamma_t}}\right)^{\frac{1}{N-1}}> 1
$$
and considering the dilated couple $(\tilde{\Omega}_t,\tilde{\Gamma}_t)\in \as(\R^N)$ with
$$
\tilde{\Omega}_t:=\eta(t)\Omega_t,\qquad\text{and}\qquad \tilde{\Gamma}_t:=\eta(t)\Gamma_t,
$$
the optimality of $(\Omega,\Gamma)$ and the admissibility of $(\tilde{\Omega}_t,\tilde{\Gamma}_t)$ for Problem \eqref{eq:pb-gen-gamma} yield 
\begin{equation}
\label{eq:lkb-dil-gen}
F_\gamma(\Omega,\Gamma)\le F_\gamma(\tilde{\Omega}_t,\tilde{\Gamma}_t).
\end{equation}
Since $\tilde \lambda_{h,\beta}(\Om_t,\Gamma_t)\to \tilde \lambda_{h,\beta}(\Om,\Gamma)$  for every $h\ge 1$ and $\eta(t)\to 1$ as $t\to +\infty$, we can assume that there exists $t_0>0$ such that for every $t\ge t_0$ there exist indexes $k\in \{1,\dots,k_\ell\}$, $i\in \{1,\dots,k_{\ell,1}\}$, and a constant $\delta>0$ such that
\begin{equation}
\label{eq:mu-min}
[\mu(\Omega,\Gamma)]_{k}<[\mu(\tilde{\Omega}_t,\tilde{\Gamma}_t)]_{k}=\tilde \lambda_{i,\beta}(\tilde{\Omega}_t,\tilde{\Gamma}_t)\le \tilde \lambda_{i,\beta}(\Omega_t,\Gamma_t)<[\mu(\Omega,\Gamma)]_{k+1}-\delta,
\end{equation}
where we mean $[\mu(\Omega,\Gamma)]_{k_\ell+1}:=[\mu(\Omega,\Gamma)]_{k_\ell}+1$ for $k=k_\ell$. Indeed, if this is not the case $(\tilde{\Omega}_t,\tilde{\Gamma}_t)$ would be a minimizer bounded in the positive direction $x_1$. Let $E_{k,i}$ be the set of those $t$ satisfying \eqref{eq:mu-min}: the sets $E_{k,i}$ cover $[t_0,+\infty[$ and are in a finite number. 
\par
From \eqref{eq:lkb-dil-gen} we infer
$$
F_\gamma(\Om_t,\Gamma_t)-F_\gamma(\tilde \Om_t,\tilde \Gamma_t)\le F_\gamma(\Om_t,\Gamma_t)-F_\gamma(\Om,\Gamma).
$$
In view of Lemma \ref{lem:cut}, and of the monotonicity and Lipschitz continuity of $f$, we have
$$
F_\gamma(\Om_t,\Gamma_t)-F_\gamma(\Om,\Gamma)\le C_1 \hn(\Om \cap \{x_1=t\})
$$
for some $C_1>0$ independent of $t$.
Moreover, if $t\in E_{k,i}$ the Lipschitz estimate from below given by property (f2) and \eqref{eq:mu-min} entail
$$
F_\gamma(\Om_t,\Gamma_t)-F_\gamma(\tilde \Om_t,\tilde \Gamma_t)\ge C [\tilde \lambda_{i,\beta}(\Om_t,\Gamma_t)-\tilde \lambda_{i,\beta}(\tilde \Om_t,\tilde \Gamma_t)],
$$
so that we get thanks to Remark \ref{rem:scaling}
\begin{multline*}
\tilde \lambda_{i,\beta}(\Om_t,\Gamma_t)\le \tilde \lambda_{i,\beta}(\tilde \Om_t,\tilde \Gamma_t)+C_2 \hn(\Om\cap \{x_1<t\})\\
\le \frac{1}{\eta(t)}\tilde \lambda_{i,\beta}(\Om_t,\Gamma_t)+C_2 \hn(\Om\cap \{x_1<t\}),
\end{multline*}
where $C_2:=C_1C^{-1}$.
Since $\tilde \lambda_{i,\beta}(\Om_t,\Gamma_t)\ge C_3$ on $[t_0,+\infty[$ for every $i=1,\dots,k_{\ell,1}$, we conclude easily that for a.e. $t\in [t_0,+\infty[$
the following inequality is satisfied
$$
\hn(\partial^*(\Omega \setminus \Om_t))\le C_4\hn(\Omega\cap\{x_1=t\}).
$$
where $C_4>0$. Following the arguments of Step 2 in the proof of Lemma \ref{lem:cut}, we infer that $\Om$ is bounded in the positive direction $x_1$, a contradiction, and the step is concluded.

\vskip10pt\noindent{\bf Step 2.}
Let us prove existence of bounded minimizers proceeding by induction on the order $k_\ell$ of the highest eigenvalue involved. 
\par
For $k_\ell=1$, since $f$ is increasing and taking into account Remark \ref{rem:1ball}, minimizers are balls of perimeter $p$.
Let us now assume that a bounded minimizer exists for functionals with $k_\ell\le k$, and let us prove it for $k_\ell\le k+1$. 
\par
Let $(\Omega_n,\Gamma_n)_{n\in\N}$ be a minimizing sequence for problem \eqref{eq:pb-gen-gamma}. We can assume that 
$$
\Prob{\Omega_n}{\Gamma_n}\to \bar p,
$$
where $\bar p>0$ is the minimal  value which can be achieved in the limit by the generalized perimeters of a minimizing sequence. The fact that $\bar p>0$ is due to the fact that otherwise we would have $|\Om_n| \to 0$, and then $F_\gamma(\Om_n,\Gamma_n)\to +\infty$ thanks to property $(f1)$. We can assume moreover
$$
\lim_{n\to+\infty}|\Omega_n|=m,\quad \text{with $0<m<+\infty$}.
$$
\par
Let us apply, as in the proof of Theorem \ref{th:main1}, a concentration-compactness argument to the sequence $(1_{\Omega_n})_{n\in\N}$. 
The vanishing case, as above, cannot occur. If compactness holds true, thanks to the monotonicity of the function $f$ 
we obtain the existence of a minimizer $(\Om,\Gamma)\in \as(\R^N)$: in view of Step 1, we can assume that it is also bounded.
\par
In order to get the conclusion, we need thus to deal with the dichotomy case. Following the arguments in the proof of  Theorem \ref{th:main1}, and using the Lipschitz continuity of $f$, we come up easily with a minimizing sequence of the form 
$$
(\Om_{n,1}\cup \Om_{n,2},\Gamma_{n,1}\cup \Gamma_{n,2})_{n\in\N},
$$
where $\Om_{n,1}$ and $\Om_{n,2}$ are well separated with
$$
\Prob{\Omega_{n,1}}{\Gamma_{n,1}}=\bar p_1\qquad \text{and}\qquad
\Prob{\Omega_{n,2}}{\Gamma_{n,2}}=\bar p_2,
$$
$\bar p_1,\bar p_2>0$ and $\bar p_1+\bar p_2=\bar p$. We know that the spectrum of $(\Omega_{n,1}\cup\Omega_{n,2} \Gamma_{n,1}\cup\Gamma_{n,2})$
is given by the union of those of $(\Omega_{n,1},\Gamma_{n,1})$ and $(\Omega_{n,2},\Gamma_{n,2})$ by the formula
$$
\tilde \lambda_{k,\beta}(\Omega_{n,1}\cup\Omega_{n,2},\Gamma_{n,1}\cup\Gamma_{n,2})=
\min_{i=0,\ldots,k}\max\left\{\tilde\lambda_{i,\beta}(\Omega_{n,1},\Gamma_{n,1}),\tilde\lambda_{k-i,\beta}(\Omega_{n,2},\Gamma_{n,2})\right\}.
$$
We observe that the eigenvalues appearing in the computation of $F_\gamma(\Omega_{n,1}\cup\Omega_{n,2},\Gamma_{n,1}\cup\Gamma_{n,2})$ involve both
$(\Omega_{n,1},\Gamma_{n,1})$ and $(\Omega_{n,2},\Gamma_{n,2})$: otherwise, if for example only those of $(\Omega_{n,1},\Gamma_{n,1})$ were involved, then $(\Omega_{n,1},\Gamma_{n,1})$ would be a minimizing sequence with a perimeter below the minimal threshold $\bar p$, which is impossible. As a consequence, up to a subsequence, we may assume that the computation $F_\gamma(\Omega_{n,1}\cup\Omega_{n,2},\Gamma_{n,1}\cup\Gamma_{n,2})$ involves the ordered $k_\ell$-tuple
$$
(\{\tilde \lambda_{i,\beta}(\Omega_{n,1},\Gamma_{n,1})\}_{i=1,\dots, k^1_{\ell,1}}, \{\tilde \lambda_{j,\beta}(\Omega_{n,2},\Gamma_{n,2})\}_{j=1,\dots, k^2_{\ell,1}}, \{\gamma_h\}_{h=1,\dots, k_{\ell,2}})^{ord}
$$
where $k^1_{\ell,1},k^2_{\ell,1}>0$ are independent of $n$ with $k^1_{\ell,1}+k^2_{\ell,1}=k_{\ell,1}$.
\par
Thanks to assumption $(f1)$, we can assume that
$$
\tilde \lambda_{i,\beta}(\Omega_{n,1},\Gamma_{n,1}) \to \delta_i>0\qquad\text{and}\qquad \tilde \lambda_{j,\beta}(\Omega_{n,2},\Gamma_{n,2}) \to \eta_j>0
$$
for every $i=1,\dots, k^1_{\ell,1}$ and $j=1,\dots,k^2_{\ell,1}$.
\par
Let $(A^*,\Gamma^*)$ be a bounded minimizer of the functional
$$
(A,\Gamma)\mapsto f((\{\tilde \lambda_{i,\beta}(A,\Gamma)\}_{i=1,\dots, k^1_{\ell,1}}, 
\{\eta_j\}_{j=1,\dots,k^2_{\ell,1}},\{\gamma_h\}_{h=1,\dots, k_{\ell,2}})^{ord})
$$
under the perimeter constraint $\bar p_1$, and let $(B^*,K^*)$  be a bounded minimizer of the functional
$$
(B,K)\mapsto f((\{\delta^*_i\}_{i=1,\dots,k^1_{\ell,1}}, \{\tilde \lambda_{j,\beta}(B,K)\}_{j=1, \dots, k^2_{\ell,1}},\{\gamma_h\}_{h=1,\dots, k_{\ell,2}})^{ord})
$$
under the perimeter constraint $\bar p_2$, where
$$
\delta^*_i:=\tilde \lambda_{i,\beta}(A^*,\Gamma^*).
$$
The existence of $(A^*,\Gamma^*)$ and $(B^*,K^*)$ is guaranteed by the induction step. It turns out that we can put the two configurations at a positive distance in order to create $(A^*\cup B^*, \Gamma^*\cup K^*)\in \as(\R^N)$ with generalized perimeter equal to $\bar p$. We get easily that $(A^*\cup B^*, \Gamma^*\cup K^*)$ is a minimizer of the problem. Indeed using the Lipschitz continuity of $f$
\begin{multline*}
\liminf_{n\to +\infty} F_\gamma(\Om_n,\Gamma_n)=\liminf_{n\to +\infty}F_\gamma(\Om_{n,1}\cup \Om_{n,2},\Gamma_{n,1}\cup \Gamma_{n,2})\\
=\liminf_{n\to +\infty}f((\{\tilde \lambda_i(\Omega_{n,1},\Gamma_{n,1})\}, \{\tilde \lambda_j(\Omega_{n,2},\Gamma_{n,2})\},\{\gamma_h\})^{ord})\\
=\liminf_{n\to +\infty}f((\{\tilde \lambda_i(\Omega_{n,1},\Gamma_{n,1})\}, \{\eta_j\}, \{\gamma_h\})^{ord})\ge
f((\{\delta^*_i\}, \{\eta_j\}, \{\gamma_h\})^{ord})\\
=
\liminf_{n\to+\infty}f((\{\delta^*_i\}, \{\tilde \lambda_{j,\beta}(\Om_{n,2},\Gamma_{n,2})\}, \{\gamma_h\})^{ord})
\\
\ge f((\{\tilde \lambda_{i,\beta}(A^*,\Gamma^*)\},\{\tilde \lambda_{j,\beta}(B^*,K^*)\}, \{\gamma_h\})^{ord})\ge F_\gamma(A^*\cup B^*,\Gamma^*\cup K^*),
\end{multline*}
the last inequality coming from the monotonicity of $f$ and the fact that, as ordered $k_\ell$-tuples,
\begin{multline*}
(\{\tilde \lambda_{k,\beta}(A^*\cup B^*,\Gamma^*\cup K^*\}_{k=1,\dots, k_{\ell,1}}, \{\gamma_h\}_{h=1,\dots, k_{\ell,2}})^{ord}\\
\le (\{\tilde \lambda_{i,\beta}(A^*,\Gamma^*\}_{i=1,\dots, k^1_{\ell,1}}, \{\tilde \lambda_{j,\beta}(B^*,K^*)\}_{j=1,\dots, k^2_{\ell,1}}, \{\gamma_h\}_{h=1,\dots, k_{\ell,2}})^{ord}.
\end{multline*}
\end{proof}

We are now ready to prove Theorem \ref{thm:main3}.

\begin{proof}[\bf Proof of Theorem \ref{thm:main3}]
The existence of bounded minimizers follow from the more general result given by Theorem \ref{thm:main3-gen}. The density issue concerning Lipschitz domains follows from the monotonicity of $f$ together with Theorem \ref{thm:den} and Remark \ref{rem:approx-lek}. 
\par
In order to conclude, we need to show that every minimizer $(\Om,\Gamma)$ is bounded. Indeed, following the arguments of Step 1 in the proof of Theorem \ref{thm:main3-gen}, since no $\gamma$ is involved, the strict monotonicity of $f$ entails that the key inequality \eqref{eq:mu-min} is always satisfied, which  yields the boundedness of $\Om$.
\end{proof}

\begin{remark}
\label{rem:pen-proof}
{\rm
The result for the perimeter penalized version of the problem given in Remark \ref{rem:penal} follows easily by slightly modifying the arguments of the proof of Theorem \ref{thm:main3-gen} and dealing with 
$$
(\Om,\Gamma)\mapsto F_\gamma(\Om,\Gamma)+\Lambda \Prob{\Om}{\Gamma}.
$$
Boundedness of {\it every} minimizer follows from the fact that minimality entails
$$
\Lambda \left[\Prob{\Om}{\Gamma}-\Prob{\Om_t}{\Gamma_t}\right] \le F(\Om_t,\Gamma_t)-F(\Om,\Gamma)\le C \hn(\Om\cap \{x_1=t\})
$$
which yields 
$$
\hn(\partial^*(\Omega \setminus \Om_t))\le C_1\hn(\Omega\cap\{x_1=t\})
$$
for a.e. $t$ large enough, from which boundedness follows.
\par
As far as existence is concerned, only the dichotomy case needs some small modifications: in particular the configurations $(A^*,\Gamma^*)$ and $(B^*,K^*)$ constructed by considering the associated problems with {\it fixed} perimeters $\bar p_1$ and $\bar p_2$ are still sufficient to get the conclusion.
}
\end{remark}

\paragraph*{\bf Acknowledgements}  
The authors S.C and A.G. have been supported in their work, respectively, by the National Research Projects ``Elliptic and parabolic problems, heat kernel estimates and spectral theory'' (PRIN 20223L2NWK) and ``Variational methods for stationary and evolution problems with singularities and interfaces'' (PRIN 2022J4FYNJ), funded by the Italian Ministry of University and Research. Both authors are members of the Gruppo Nazionale per l’Analisi Matematica, la Probabilit\`a e le loro Applicazioni (GNAMPA) of the Istituto Nazionale di Alta Matematica (INdAM). S.C. acknowledges the support of the INdAM - GNAMPA 2023 Project ``Problemi variazionali per funzionali e operatori non-locali''.
\par
This manuscript has no associated data.

\bibliographystyle{plain}
\bibliography{agbiblio.bib}
\end{document}